\documentclass[aihp]{imsart}
\usepackage{comment}
\RequirePackage{amsthm,amsmath,amsfonts,amssymb}
\RequirePackage[numbers]{natbib}
\RequirePackage{booktabs,tabularx,array,longtable}
\RequirePackage{enumitem,aliascnt}
\RequirePackage{microtype}
\RequirePackage[hidelinks]{hyperref}
\RequirePackage[nameinlink,capitalise]{cleveref}
\RequirePackage{graphicx}
\RequirePackage{float}
\RequirePackage{tikz}
\usetikzlibrary{arrows.meta,positioning,calc,fit,backgrounds}

\startlocaldefs
\numberwithin{equation}{section}
\theoremstyle{plain}
\newtheorem{theorem}{Theorem}[section]
\newaliascnt{proposition}{theorem}
\newtheorem{proposition}[proposition]{Proposition}
\aliascntresetthe{proposition}
\newaliascnt{lemma}{theorem}
\newtheorem{lemma}[lemma]{Lemma}
\aliascntresetthe{lemma}
\newaliascnt{corollary}{theorem}
\newtheorem{corollary}[corollary]{Corollary}
\aliascntresetthe{corollary}
\theoremstyle{definition}
\newaliascnt{assumption}{theorem}

\aliascntresetthe{assumption}
\newaliascnt{definition}{theorem}

\aliascntresetthe{definition}
\newaliascnt{remark}{theorem}
\newtheorem{remark}[remark]{Remark}
\aliascntresetthe{remark}

\crefname{assumption}{Assumption}{Assumptions}
\Crefname{assumption}{Assumption}{Assumptions}
\crefname{definition}{Definition}{Definitions}
\Crefname{definition}{Definition}{Definitions}
\crefname{theorem}{Theorem}{Theorems}
\Crefname{theorem}{Theorem}{Theorems}
\crefname{proposition}{Proposition}{Propositions}
\Crefname{proposition}{Proposition}{Propositions}
\crefname{lemma}{Lemma}{Lemmas}
\Crefname{lemma}{Lemma}{Lemmas}
\crefname{corollary}{Corollary}{Corollaries}
\Crefname{corollary}{Corollary}{Corollaries}
\crefname{remark}{Remark}{Remarks}
\Crefname{remark}{Remark}{Remarks}

\newcommand{\R}{\mathbb R}
\newcommand{\cM}{\mathcal M}
\newcommand{\supp}{\operatorname{supp}}
\newcommand{\Lip}{\operatorname{Lip}}
\newcommand{\diag}{\operatorname{diag}}

\newcommand{\TV}{\mathrm{TV}}
\newcommand{\1}{\mathbf 1}
\newcommand{\norm}[1]{\left\lVert #1\right\rVert}
\newcommand{\abs}[1]{\left\lvert #1\right\rvert}
\providecommand{\kwdgroup}[2][]{#2}
\newcolumntype{Y}{>{\raggedright\arraybackslash}X}
\newcolumntype{L}[1]{>{\raggedright\arraybackslash}p{#1}}
\allowdisplaybreaks

\definecolor{proofink}{RGB}{40,44,52}
\definecolor{proofmuted}{RGB}{88,92,101}
\definecolor{proofblue}{RGB}{42,91,154}
\definecolor{proofbluebg}{RGB}{239,245,253}
\definecolor{proofgreen}{RGB}{47,122,91}
\definecolor{proofgreenbg}{RGB}{240,248,244}
\definecolor{proofpurple}{RGB}{102,82,145}
\definecolor{proofpurplebg}{RGB}{245,241,250}
\definecolor{proofgold}{RGB}{153,113,38}
\definecolor{proofgoldbg}{RGB}{253,248,236}
\definecolor{proofred}{RGB}{157,69,59}
\definecolor{proofredbg}{RGB}{253,242,240}
\definecolor{proofgraybg}{RGB}{248,249,251}
\endlocaldefs

\usepackage{tikz}
\usepackage{amsmath,amssymb,mathtools}
\usetikzlibrary{arrows.meta,positioning,calc,fit,backgrounds}

\definecolor{ink}{RGB}{40,44,52}
\definecolor{muted}{RGB}{88,92,101}
\definecolor{blue}{RGB}{42,91,154}
\definecolor{bluebg}{RGB}{239,245,253}
\definecolor{green}{RGB}{47,122,91}
\definecolor{greenbg}{RGB}{240,248,244}
\definecolor{purple}{RGB}{102,82,145}
\definecolor{purplebg}{RGB}{245,241,250}
\definecolor{gold}{RGB}{153,113,38}
\definecolor{goldbg}{RGB}{253,248,236}
\definecolor{red}{RGB}{157,69,59}
\definecolor{redbg}{RGB}{253,242,240}
\definecolor{graybg}{RGB}{248,249,251}
\begin{document}
\begin{frontmatter}
\title{\vspace{-0.6in}An AI-Assisted Solution to the Signed BAR Conjecture: Uniqueness in the Harrison--Reiman Class and a Completely-$\mathcal{S}$ Class Obstruction}
\runtitle{Signed BAR uniqueness conjecture}
\runauthor{Y. Lu and Y. Zhu}

\begin{aug}
\author[A]{\fnms{Yiping}~\snm{Lu}}
\author[A]{\fnms{Youheng}~\snm{Zhu}}
\address[A]{Department of Industrial Engineering and Management Sciences,
McCormick School of Engineering, Northwestern University.}
\end{aug}

\begin{abstract}
For a multidimensional reflected diffusion, determining whether the associated basic adjoint relationship (BAR) uniquely characterizes the stationary distribution is a basic uniqueness problem in the BAR approach. The problem has remained unresolved for more than 35 years since the introduction of the BAR approach. In this paper, we resolve the finite-signed uniqueness problem for stable Harrison--Reiman data with a nonsingular $M$-matrix reflection matrix. The proof uses pathwise differentiability of the reflected diffusion implies feasible directional differentiability of the probabilistic resolvent to show that, at boundary points, its one-sided initial-state derivative factors through the tangent projection and vanishes along active reflection directions. An interior one-sided convolution then yields smooth test functions whose oblique derivatives are uniformly bounded and converge pointwise to zero on each closed face. The interior signed measure is consequently invariant for the reflected semigroup. A Jordan-decomposition argument identifies it as a scalar multiple of the unique invariant probability, and an induction over boundary strata, using invertibility of the principal reflection blocks, identifies the boundary measures. The proof was discovered with the assistance of ChatGPT 5.5 Pro and subsequently verified by the authors.

We also show that the nonsingular $M$-matrix assumption is structural. In the larger completely-$\mathcal{S}$ class, a nonsingular reflection matrix with a singular proper principal block admits boundary gauges supported on lower-dimensional strata. Under standard exponential ergodicity and a mild one-step regulator bound, these gauges produce nonzero zero-mass signed BAR tuples; indeed the zero-mass interior BAR coordinates contain an infinite-dimensional subspace. A four-parameter three-dimensional family, including an explicit rational example, verifies the obstruction. Thus the finite signed version of the Dai--Dieker question has a positive answer in the Harrison--Reiman $M$-matrix class and a negative answer in a natural completely-$\mathcal{S}$ extension.
\end{abstract}


\end{frontmatter}

\section{Introduction}

Semimartingale reflected Brownian motions (SRBMs) in the nonnegative orthant are diffusion approximations for stochastic networks in heavy traffic.  In the interior of the orthant the process behaves as a Brownian motion with drift and covariance matrix; when it reaches a face, it is pushed back into the state space in an oblique direction prescribed by the corresponding column of a reflection matrix.  The Harrison--Reiman construction \cite{HarrisonReiman1981,HarrisonReiman1981Distribution} is the canonical orthant model behind open queueing networks in heavy traffic \cite{Reiman1984,harrison1987brownian,HarrisonNguyen1993,Harrison1985,Williams1995}; it is the main positive setting of this paper.

A central analytic object for such reflected diffusions is the basic adjoint relationship (BAR).  It appears in the early stationary analysis and product-form theory for RBM/SRBM \cite{HarrisonReiman1981Distribution,harrison1987multidimensional,harrison1987brownian}, underlies numerical methods for orthant SRBMs \cite{DaiHarrison1991,DaiHarrison1992}, has been used in steady-state heavy-traffic approximation through the BAR approach \cite{BravermanDaiMiyazawa2017,BravermanDaiMiyazawa2025}, and is one of the standard weak formulations used to characterize stationary distributions of reflected diffusions \cite{DaiDieker2011,KangRamanan2014}. If $\pi$ is an interior measure and $\nu_i$ is a boundary measure on the face $F_i=\{x_i=0\}$, the BAR has the form
\begin{equation}\label{eq:intro-BAR}
 \int_E Lf\,d\pi+
 \sum_{i=1}^d\int_{F_i}D_i f\,d\nu_i=0,
 \qquad f\in C_b^2(E),
\end{equation}
where $L$ is the interior diffusion generator and $D_i$ is the directional derivative in the $i$th reflection direction.  The stationary distribution $\pi_0$, together with its stationary boundary occupation measures $\nu_i^0$, satisfies \eqref{eq:intro-BAR}.  The basic uniqueness question is whether the converse holds:
\begin{center}
    \emph{does a BAR solution necessarily have interior part equal to the stationary distribution? }
\end{center}

The issue has persisted for more than three decades, remaining an open problem since the inception of the BAR approach. The open problem was first stated as a conjecture in  \cite{DaiHarrison1991} for SRBMs in a two dimensional rectangle and in \cite{DaiHarrison1992} for SRBMs in a $d$-dimensional orthant. Dai and Dieker \cite{DaiDieker2011} describe the fundamental open problem concerning the Basic Adjoint Relationship (BAR) for multidimensional diffusion processes. Specifically, for both Semimartingale Reflecting Brownian Motions (SRBMs) and piecewise Ornstein–Uhlenbeck (OU) processes. Dai and Dieker \cite[Proposition~1 and Open Problem~1]{DaiDieker2011} formulated the BAR characterization with bounded $C^2$ tests, proved the corresponding characterization in the positive-measure setting, and asked for the signed analogue.  The compactly supported $C^2$ formulation leads to the same finite-signed uniqueness problem.  The bounded-test identity immediately implies the compactly supported one.  Conversely, let $f\in C_b^2(E)$ and choose $\chi_n\in C_c^\infty(\mathbb R^d)$ with $0\le\chi_n\le1$, $\chi_n=1$ on $\{|x|\le n\}$, and $\|\nabla\chi_n\|_\infty+\|D^2\chi_n\|_\infty\to0$.  Applying the compactly supported identity to $\chi_n f$ and expanding $L(\chi_n f)$ and $D_i(\chi_n f)$ gives the bounded-test identity after passage to the limit, because $\chi_n\to1$ pointwise and all error terms are uniformly bounded by constants times $\|\nabla\chi_n\|_\infty+\|D^2\chi_n\|_\infty$ against finite signed measures.  Throughout the paper we therefore use the bounded-test class $C_b^2(E)$, which is the formulation needed to insert the one-sided smoothings of the probabilistic resolvent without an artificial spatial cutoff.

In the signed problem one allows $\pi$ and the $\nu_i$ to be finite signed measures.  The question then becomes linear: is every finite signed BAR tuple a scalar multiple of the stationary tuple?  This signed formulation is more delicate than the positive one.  Positive recurrence identifies invariant probabilities, but the BAR permits cancellation between signed interior and boundary terms.  Moreover, the natural functions that identify invariant measures are probabilistic resolvents, which are not classical BAR tests at the corners.

\subsection*{Related work}

\paragraph{BAR characterization of stationary probabilities} As shown in the the original BAR calculations for SRBMs \cite{harrison1987multidimensional,harrison1987brownian},  positive-measure BAR characterizations identify stationary probabilities, and in many formulations also the associated boundary occupation measures, once the reflected diffusion and its stationary regime are already well posed  \cite{DaiHarrison1991,DaiHarrison1992,KangRamanan2014}.   These results do not, by themselves, exclude sign-changing finite measures whose interior and boundary terms cancel in the BAR.  Our positive theorem addresses exactly that finite-signed nullspace question in the stable Harrison--Reiman nonsingular-$M$-matrix class, and it identifies the full boundary tuple as well as the interior coordinate.

Much of the stationary SRBM literature concerns explicit formulas, transforms, asymptotics, or numerical computation rather than signed uniqueness.  Product-form and skew-symmetry results originate with Harrison and Williams \cite{harrison1987multidimensional}; numerical and approximation methods based on the BAR go back at least to Dai and Harrison \cite{DaiHarrison1991,DaiHarrison1992} and continue in the steady-state heavy-traffic BAR approach for queueing networks \cite{BravermanDaiMiyazawa2017,BravermanDaiMiyazawa2025}; two-dimensional and wedge analyses have been developed through sum-of-exponentials, geometric, and boundary-value/functional-equation methods \cite{DiekerMoriarty2009,DaiMiyazawa2011,DaiMiyazawa2013,FranceschiRaschel2017,FranceschiRaschel2019}.  The present proof uses none of these explicit analytic representations.  Its role is instead structural: it proves that, in the stated $M$-matrix class, the finite signed BAR has no hidden zero-mass directions.

\paragraph{Skorokhod-map Differentiability} Lipschitz, convex-duality and differentiability properties of oblique reflection maps were developed in deterministic form by Dupuis--Ishii, Dupuis--Ramanan, Mandelbaum--Ramanan, and Lipshutz--Ramanan \cite{DupuisIshii1991,DupuisRamanan1999I,DupuisRamanan1999II,MandelbaumRamanan2010,LipshutzRamanan2018}.  We use the reflected-diffusion version of this theory, namely the pathwise differentiability and sensitivity results of Lipshutz and Ramanan \cite{LipshutzRamanan2019,LipshutzRamanan2021}, only after verifying their assumptions for the normalized Harrison--Reiman data.  The negative result is complementary to the existence and stability literature for completely-$\mathcal{S}$ data: Taylor--Williams and Dai--Williams give the relevant SRBM existence frameworks \cite{TaylorWilliams1993,DaiWilliams1995}, while Lyapunov and recurrence criteria for SRBMs are developed for example in \cite{DupuisWilliams1994,BramsonDaiHarrison2010,Sarantsev2017}.  Section~\ref{sec:sharpness} shows that existence and recurrence alone do not replace invertibility of every active principal block.

\subsection*{Technical Overview}

Our positive result answers the signed Dai--Dieker problem for stable Harrison--Reiman data with a nonsingular $M$-matrix reflection matrix.  The proof is organized around a resolvent invariant identity.  Let $R_\lambda h=\int_0^\infty e^{-\lambda t}P_t h\,dt$ be the probabilistic resolvent of the reflected semigroup.  Our core contribution is proving the fact that every finite signed BAR tuple satisfies
\begin{equation}\label{eq:intro-RI}
\tag{RI}
 \int_E(\lambda R_\lambda h-h)\,d\bar\pi=0,
 \qquad h\in C_0(E),\quad \lambda>0.
\end{equation}
This identity says exactly that the interior signed measure is invariant under the reflected semigroup.  Indeed, using $R_\lambda h=\int_0^\infty e^{-\lambda t}P_th\,dt$, \eqref{eq:intro-RI} says that the Laplace transform of $t\mapsto\bar\pi(P_th)-\bar\pi(h)$ vanishes for every $h\in C_0(E)$.  Strong continuity of the Feller semigroup upgrades this to $\bar\pi P_t=\bar\pi$ for all $t\ge0$.  If $\bar\pi=\bar\pi^+-\bar\pi^-$ is the Jordan decomposition, positivity of the Markov kernel gives $|\bar\pi P_t|\le |\bar\pi|P_t$; equality of total masses then makes $|\bar\pi|$ invariant, and hence both Jordan components are invariant positive finite measures.  After normalization, every nonzero component is an invariant probability, so uniqueness of the invariant probability gives $\bar\pi=c\pi_0$.  Subtracting $c$ times the stationary BAR leaves a pure boundary identity, and the nonsingular principal reflection blocks identify the boundary measures by an induction over strata.

The only nontrivial point in this chain is the derivation of \eqref{eq:intro-RI}.  Formally, if $g=R_\lambda h$ were an admissible $C_b^2$ test satisfying $D_i g=0$ on $F_i$, then \eqref{eq:intro-RI} would follow by inserting $g$ into the BAR and using $(\lambda-L)g=h$.  This formal argument is misleading because at corners the resolvent need not be a classical $C^2$ function on the closed orthant; \cref{app:resolvent-not-C2} gives a stable Harrison--Reiman example where such $C^2$ regularity is impossible.  The proof therefore works in the topology actually seen by finite signed measures: uniform convergence of the interior equation and vanishing of the boundary terms after integration against arbitrary finite signed boundary measures.

The approximation used in the proof is intentionally simple.  We do not insert
$g=R_\lambda h$ itself into the BAR.  Instead we replace it by the one-sided
smoothing
\[
        g_\varepsilon(x)=\int \rho(w)g(x+\varepsilon w)\,dw .
\]
The mollifier is supported strictly inside the positive orthant, so the value of
$g_\varepsilon(x)$ only uses values of $g$ at interior points $x+\varepsilon w$.
This smoothing supplies the required bounded $C^2$ regularity for each fixed
$\varepsilon$.  The only delicate point is to show that these legitimate
$C_b^2$ tests have asymptotically zero boundary contribution.  The projected
boundary derivative of the resolvent gives
\[
        D_i g_\varepsilon(x)\longrightarrow 0,\qquad x\in F_i,
\]
with a uniform bound sufficient for dominated convergence against an arbitrary
finite signed boundary measure.  Thus the functions $g_\varepsilon$ approximate
the resolvent in exactly the topology seen by the BAR: the interior equation
converges to $(\lambda-L)R_\lambda h=h$, while all boundary terms vanish.

The paper also explains why the $M$-matrix hypothesis is not merely a proof artifact.  In the completely-$\mathcal{S}$ existence class, a singular proper principal block may cancel all active normal components of a boundary gauge supported on a lower-dimensional stratum.  The remaining tangential derivative produces a centered interior source.  Under a quantitative recurrence assumption, the zero potential of this source gives a nonzero signed BAR tuple with zero interior mass.  Thus signed uniqueness fails in a natural completely-$\mathcal{S}$ extension.

\subsection*{The Role of AI-assistance} The proof given here was not produced by an AI system in a single pass; it is the outcome of an extended, human-directed collaboration (for 3 weeks) in which large language models served as an exploratory and organizational aid, while every mathematical decision and all verification rested with the authors. By shifting the focus from merely verifying the conjecture to characterizing the specific domain where it holds, this study not only reveals the essential divergence between Harrison-Reiman Class and Completely-$\mathcal{S}$ Class but also demonstrates the vital role of human-AI collaboration in advancing complex mathematical research. Following the program in Dai and Dieker's open-problem note \cite{DaiDieker2011}, we first attacked uniqueness in the completely-$\mathcal{S}$  class, where the crux is the low regularity of the solution at the boundary. Over many rounds of interaction the model carried out the boundary-layer expansion and tested whether the boundary contribution is sign-definite and whether it can be absorbed by the interior solution. When this cancellation repeatedly failed for $d>3$, the authors chose to abandon the direct route and to construct a counterexample in the singular regime; the construction presented here is our own, and it delimits the regime in which signed uniqueness can be expected. We then turned to signed-measure uniqueness in the Harrison–Reiman class. Our first attempt proceeded through a Kato-type inequality, where the obstruction is the boundary term produced by the integration by parts; to organize the inductive cancellation of this term across the boundary strata, we prompted the model to adopt a homological-algebra–style bookkeeping. This yielded a long (roughly 150-page, see \url{https://drive.google.com/file/d/1QEMTMYR9d0l3ToJtdVHEeYT9TF5Cudui/view?usp=sharing}) proof outline that passed an initial screening by an ensemble of ten independent model/agent reviewers. Such consensus is not a proof, and we treated it only as a filter: the argument was subsequently checked by the authors, conclusion by conclusion, with each regularity hypothesis verified for mutual consistency. In the course of this verification the model surfaced the pathwise-differentiability results of Lipshutz and Ramanan \cite{LipshutzRamanan2018}, which considerably simplified the argument and, after further iteration, produced the proof in its present form. The authors have verified every step and are solely responsible for the correctness of the results. Additionally, we attempted to generate a positive proof via one-shot prompting, leveraging the premise that the conjecture holds true within the Harrison-Reiman class. However, both ChatGPT 5.5 Pro-extended and Claude Opus 4.8 max failed this task. The chat logs are available at: \url{https://chatgpt.com/share/6a44a502-d034-83ea-9608-eecb9ecc898d} and \url{https://claude.ai/share/25a16238-360a-4649-935f-b23b4ec500ff}(Attempts \url{https://chatgpt.com/share/6a44b084-91dc-83ea-8fc2-49b06770025d} to solve the problem, even when prompted with the literature \cite{LipshutzRamanan2018, LipshutzRamanan2019}, proved unsuccessful.). Surprisingly, contemporary AI approaches even fail to leverage the specific properties of the Harrison–Reiman class, which are essential for the proof of positivity established via the counterexample in the general Completely-$\mathcal{S}$ class presented in this paper. We hypothesize that the AI derived meaningful insights from the first 150 pages version of computations, even though these results were not explicitly incorporated into the final proof. This outcome highlights the potential of AI assistance in tackling open mathematical problems, while simultaneously underscoring the indispensable role of human verification and guidance throughout the process.
\subsection*{Organization of the Paper} 
We organize the paper as follows: \Cref{sec:setting} states the SRBM and BAR setting, states the main theorem, and reduces the proof to the
resolvent identity \eqref{eq:intro-RI}.  \Cref{sec:resolvent-regularity}
establishes the two technical properties of $g=R_\lambda h$ needed later for the approximation:
the interior resolvent equation and the projected boundary derivative that will
make $D_i g_\varepsilon$ vanish on $F_i$.  \Cref{sec:smoothing} carries out the
one-sided smoothing construction, inserts $g_\varepsilon\in C_b^2(E)$ directly
into the BAR, and proves \eqref{eq:intro-RI}.  \Cref{sec:completion}
proves the implication deferred in \Cref{sec:setting}: the identity \eqref{eq:intro-RI} implies the signed BAR uniqueness conjecture, thus finishing the proof of the main theorem.  \Cref{sec:sharpness} explains why the nonsingular $M$-matrix condition
is structural by giving the completely-$\mathcal{S}$ obstruction and an explicit
three-dimensional family.  \Cref{sec:unified-understanding} repackages the positive and negative arguments through a common BAR homotopy lemma and separates the remaining issue into local boundary algebra.

\section{Setting, main theorem, and reduction to the resolvent identity}\label{sec:setting}

This section fixes the data, states the signed-measure theorem, and isolates the central reduction. The conversion of the present standing assumptions into the hypotheses of the reflected-diffusion results is carried out inline, at the point of use, inside the proof of \cref{thm:regularity-SRBM}: there each source hypothesis is recalled in the present orthant specialization and verified.

\subsection{Notation and standing conventions}\label{sec:notation}

Let $J=\{1,\ldots,d\}$, $E=\mathbb R_+^d$, and $E^\circ=(0,\infty)^d$.  For $i\in J$ write $ F_i=\{x\in E:x_i=0\}.$ For nonempty $A\subset J$, define the relative boundary stratum $ S_A=\{x\in E:x_i=0\ (i\in A),\ x_j>0\ (j\notin A)\}.$ The sets $S_A$ form a disjoint Borel decomposition of $\partial E$.

For a locally compact space $B$, $C_0(B)$ denotes the continuous real-valued functions vanishing at infinity, and $\cM(B)$ denotes the finite signed Radon measures on $B$.  For $\eta\in\cM(B)$, $|\eta|$ is its total variation measure and $\|\eta\|_{\TV}=|\eta|(B)$.  We write $\supp\eta$ for the support of a measure and $\supp f$ for the support of a function.  The symbol $\mathbf 1_B$ denotes the indicator of a set $B$.

We use the closed-domain $C^2$ convention.  Thus $C^2(E)$ consists of functions $f:E\to\R$ such that $f\in C^2(E^\circ)$ and all partial derivatives $\partial^\alpha f$, $|\alpha|\le2$, extend continuously from $E^\circ$ to $E$.  The class $C_c^2(E)$ consists of the functions in $C^2(E)$ with compact support as a subset of $E$.  The class $C_b^2(E)$ consists of the functions in $C^2(E)$ for which $f$, $\nabla f$ and $D^2f$ are bounded.  Since $E$ is the orthant, this closed-domain convention is equivalent to saying that every $f\in C^2(E)$ is the restriction to $E$ of some $F\in C^2(U)$ on an open neighborhood $U\supset E$.  For open subsets of Euclidean space, $C_c^\infty$ has its usual meaning. For the reflected semigroup we write $P_t h(x)=\mathbb E[h(Z_t^x)]$ and $ R_\lambda h(x)=\int_0^\infty e^{-\lambda t}P_t h(x)\,dt,\; \lambda>0,$ whenever the integral is finite.
We call the semigroup $P_t$ Feller if $(P_t)_{t\ge 0}$ satisfies $P_tC_0(E)\subset C_0(E)$, and is strongly continuous, i.e. $\|P_t h- h\|_\infty\to 0$ as $t\downarrow 0$ for all $h\in C_0(E)$.

\subsection{SRBM, BAR, and finite signed BAR tuples}

A semimartingale reflected Brownian motion in $E$ is specified by a drift vector $\mu\in\mathbb R^d$, a symmetric positive definite covariance matrix $\Sigma$, and a reflection matrix $R=(R_1,\ldots,R_d)$ whose $i$th column is the direction of reflection on $F_i$. Put $Q=\Sigma/2$ and
\[
 Lf=\mu\cdot\nabla f+Q:D^2f,
 \qquad
 D_i f=R_i\cdot\nabla f.
\]

\phantomsection\label{ass:HR}
Throughout the positive part of the paper we work under the following stable nonsingular $M$-matrix data. The covariance matrix $\Sigma$ is symmetric positive definite. The reflection matrix $R$ satisfies
\begin{equation}\label{eq:Mmatrix}
 R_{ii}>0,\qquad R_{ij}\le0\ (i\ne j),\qquad R^{-1}\ge0.
\end{equation}
The drift satisfies
\begin{equation}\label{eq:stability}
 R^{-1}\mu<0
\end{equation}
componentwise.
The phrase ``stable'' in this paper means exactly \eqref{eq:stability}. The linear-algebra consequences of \eqref{eq:Mmatrix} are proved in \cref{lem:Mmatrix}; the stochastic consequences used later are stated in \cref{thm:regularity-SRBM} and justified in its proof, where every source hypothesis is recalled and checked.

\phantomsection\label{def:signed-BAR}

A finite signed BAR tuple is a tuple $(\bar\pi,\bar\nu_1,\ldots,\bar\nu_d)\in\cM(E)\times\prod_{i=1}^d\cM(F_i)$ of finite signed Radon measures satisfying
\begin{equation}\label{eq:signed-BAR}
 \int_E Lf\,d\bar\pi+
 \sum_{i=1}^d\int_{F_i}D_i f\,d\bar\nu_i=0,
 \qquad f\in C_b^2(E).
\end{equation}
The stationary regulator defines finite boundary occupation measures $\nu_i^0$, and the stationary BAR is
\begin{equation}\label{eq:intro-stationary-BAR}
 \int_E Lf\,d\pi_0+\sum_{i=1}^d\int_{F_i}D_i f\,d\nu_i^0=0,
 \qquad f\in C_b^2(E).
\end{equation}

Under \eqref{eq:Mmatrix}--\eqref{eq:stability}, the normalized reflection matrix
is of Harrison--Reiman form, and the associated deterministic Skorokhod problem
drains to the origin.  Hence \cite[Theorem~2.6]{DupuisWilliams1994} and \cite[Theorem~3.5]{LipshutzRamanan2021} gives provides the existences and the uniqueness of stationary distribution \(\pi_0\) to the SRBM. Obviously, the stationary distribution and the finite stationary boundary measure characterized by the following \cref{prop:stationary-BAR} together provide a solution to the BAR equation \eqref{eq:intro-stationary-BAR}.

\begin{proposition}[Finite stationary boundary measures and stationary BAR]\label{prop:stationary-BAR}
Start the SRBM with $Z_0\sim\pi_0$ and write it in the original normalization as
\begin{equation}\label{eq:SRBM-decomposition}
 Z_t=Z_0+\mu t+\Sigma^{1/2}W_t+RY_t,
\end{equation}
where each $Y_i$ is continuous, nondecreasing, starts from zero, and increases only on $F_i$. Define, for Borel $B\subset F_i$,
\begin{equation}\label{eq:stationary-boundary-measure}
 \nu_i^0(B)=\mathbb E_{\pi_0}\int_0^1\mathbf 1_B(Z_s)\,dY_i(s).
\end{equation}
Then each $\nu_i^0$ is a finite measure supported on $F_i$, and \eqref{eq:intro-stationary-BAR} holds.
\end{proposition}

\begin{proof}
Let $a=R^{-T}\mathbf 1$. Since $R^{-1}\ge0$ and no column of the invertible matrix $R^{-1}$ is zero, $a>0$; moreover $R^Ta=\mathbf 1$. For $\alpha>0$, set
\[
 \Phi_\alpha(x)=-\sum_{k=1}^d a_k e^{-\alpha x_k}.
\]
The function and its first two derivatives are bounded. If $x\in F_i$, then, using $R_{ki}\le0$ for $k\ne i$, $x_i=0$, and $e^{-\alpha x_k}\le1$,
\[
 D_i\Phi_\alpha(x)
 =\alpha\sum_{k=1}^dR_{ki}a_ke^{-\alpha x_k}
 \ge\alpha\sum_{k=1}^dR_{ki}a_k=\alpha.
\]
It\^o's formula on $[0,1]$ gives, pathwise,
\[
 \Phi_\alpha(Z_1)-\Phi_\alpha(Z_0)
 =\int_0^1L\Phi_\alpha(Z_s)\,ds
 +M_1
 +\sum_{i=1}^d\int_0^1D_i\Phi_\alpha(Z_s)\,dY_i(s),
\]
where $M$ is a square-integrable martingale because $\nabla\Phi_\alpha$ is bounded. The first two terms on the right and the left side are integrable. The boundary sum is nonnegative, so the identity itself shows that it is integrable. Taking expectations and using stationarity therefore yields
\[
 \alpha\sum_{i=1}^d\mathbb E_{\pi_0}Y_i(1)
 \le -\mathbb E_{\pi_0}\int_0^1L\Phi_\alpha(Z_s)\,ds
 \le \|L\Phi_\alpha\|_\infty.
\]
Thus \eqref{eq:stationary-boundary-measure} is finite. Its support is contained in $F_i$ because $Y_i$ increases only there. Finally, apply It\^o's formula to $f\in C_b^2(E)$. The boundedness of $f$, $\nabla f$ and $D^2f$ makes the Brownian and drift terms integrable on $[0,1]$, and the boundary integrals are integrable by the preceding estimate. Stationarity gives
\[
 0=\int_ELf\,d\pi_0+\sum_{i=1}^d\int_{F_i}D_if\,d\nu_i^0.
\]
\end{proof}

Although \cref{prop:stationary-BAR} provides $(\pi_0,\nu_1^0,\dots,\nu_d^0)$ as a solution to the BAR equation, it remains an open question \emph{whether the BAR uniquely characterizes the stationary distribution of the diffusion process}.

\subsection{Signed BAR uniqueness in the Harrison-Reiman Class} In the Harrison-Reiman Class, \emph{i.e.} under the standing assumptions \eqref{eq:Mmatrix}--\eqref{eq:stability}, we show that the associated BAR uniquely characterizes the
stationary distribution of the diffusion process.

\begin{theorem}[Signed BAR uniqueness]\label{thm:main}
Under the standing assumptions \eqref{eq:Mmatrix}--\eqref{eq:stability}, let $\pi_0$ and $(\nu_i^0)_{i=1}^d$ be the stationary distribution of the process and the corresponding boundary measure constructed in \cref{prop:stationary-BAR}. Every finite signed BAR tuple is a scalar multiple of the stationary BAR tuple. More precisely, if \eqref{eq:signed-BAR} holds, then there exists $c\in\mathbb R$ such that
\begin{equation}\label{eq:main-conclusion}
 \bar\pi=c\pi_0,
 \qquad
 \bar\nu_i=c\nu_i^0,
 \quad i=1,\ldots,d.
\end{equation}
Consequently the vector space of finite signed BAR tuples is one-dimensional.
\end{theorem}

To prove uniqueness of finite signed BAR tuples, we first show that every BAR tuple satisfies a resolvent identity \eqref{eq:RI}; we call this identity \emph{resolvent insertion}. The resolvent insertion identity implies invariance of the interior signed measure under the reflected semigroup, and hence \(\bar{\pi}=c\pi_0\). After subtracting the interior stationary BAR, the remaining identity is purely on the boundary, and pure boundary injectivity gives \(\bar{\nu}_i=c\nu_i^0\) for all \(i=1,\ldots,d\).

\begin{proposition}[Resolvent identity criterion]\label{prop:RI-criterion}
Assume that for every finite signed BAR tuple, every $h\in C_0(E)$, and every $\lambda>0$,
\begin{equation}\label{eq:RI}
\tag{RI}
 \int_E(\lambda R_\lambda h-h)\,d\bar\pi=0.
\end{equation}
Then the conclusion of \cref{thm:main} holds.
\end{proposition}

The proof of \cref{prop:RI-criterion} is given in \cref{sec:completion}. 

\paragraph{Why the resolvent insertion \eqref{eq:RI} should hold} The reason for targeting \eqref{eq:RI} is transparent from the classical Neumann calculation.  Let $g=R_\lambda h$.  If $g$ were an admissible $C_b^2$ test and if it satisfied $D_i g=0$ on $F_i$ for $i=1,\ldots,d$, then inserting $g$ into the BAR would give
\[
0=\int_E Lg\,d\bar\pi+
  \sum_i\int_{F_i}D_i g\,d\bar\nu_i
 =\int_E Lg\,d\bar\pi.
\]
The resolvent equation $(\lambda-L)g=h$ would therefore imply $ \int_E(\lambda R_\lambda h-h)\,d\bar\pi=0.$ This is only an informal guide.  The closed-domain $C^2$ regularity required for this insertion may fail even in the stable Harrison--Reiman class.  \Cref{app:resolvent-not-C2} gives an explicit stable nonsingular $M$-matrix example and a smooth compactly supported $h$ for which $R_\lambda h\notin C^2(E)$.  The proof below therefore does not try to show that the resolvent belongs to a classical oblique-Neumann core.

\subsection{Making the resolvent insertion rigorous}

The replacement for the formal insertion is a measure-level Neumann approximation.  For smooth compactly supported $h$ we construct tests $g_\varepsilon\in C_b^2(E)$ such that, as $\varepsilon\downarrow0$,
\[
 g_\varepsilon\to R_\lambda h,
 \qquad
 (\lambda-L)g_\varepsilon\to h,
\]
against every finite signed interior measure, while
\[
 \int_{F_i}D_i g_\varepsilon\,d\bar\nu_i\to0,
 \qquad i=1,\ldots,d,
\]
for every finite signed boundary measure.  This is exactly what is needed to pass to the limit in the BAR.  The convergence is not a pointwise assertion that $R_\lambda h$ admits a classical oblique derivative $D_iR_\lambda h$ on $F_i$; it is an assertion that the boundary pairings seen by the BAR vanish.  \Cref{prop:measure-Neumann} gives the precise statement, and a density argument then extends \eqref{eq:RI} from smooth compactly supported $h$ to all $h\in C_0(E)$.  The next two sections supply the projected derivative input and the one-sided smoothing construction.

\Cref{fig:proof-architecture} summarizes where this approximation sits in the proof: the analytic work proves the target resolvent identity, while the remaining steps are the soft semigroup and boundary-identification arguments.

\begin{figure}[!htbp]
\centering
\resizebox{\textwidth}{!}{%
\begin{tikzpicture}[
  x=1cm,y=1cm,
  font=\footnotesize,
  text=proofink,
  >=Latex,
  box/.style={
    draw=proofink!48,
    rounded corners=3pt,
    line width=.55pt,
    fill=white,
    align=center,
    inner xsep=6pt,
    inner ysep=6pt,
    execute at begin node={\hyphenpenalty=10000\exhyphenpenalty=10000\emergencystretch=1.5em}
  },
  need/.style={box, draw=proofpurple!85!black, fill=proofpurplebg, text width=5.65cm, minimum height=2.15cm},
  bridge/.style={box, draw=proofpurple!85!black, fill=proofpurplebg, text width=7.40cm, minimum height=2.95cm, font=\scriptsize},
  inv/.style={box, draw=proofblue!78!black, fill=proofbluebg, text width=5.50cm, minimum height=2.15cm},
  why/.style={box, draw=proofblue!78!black, fill=proofbluebg, text width=6.00cm, minimum height=1.55cm},
  finalbox/.style={box, draw=proofgold!85!black, fill=proofgoldbg, text width=8.80cm, minimum height=1.00cm},
  bad/.style={box, draw=proofred!80!black, fill=proofredbg, text width=5.90cm, minimum height=2.05cm},
  good/.style={box, draw=proofgreen!78!black, fill=proofgreenbg, text width=5.90cm, minimum height=2.05cm},
  action/.style={box, draw=proofink!45, fill=proofgraybg, text width=13.30cm, minimum height=1.55cm, font=\scriptsize},
  arr/.style={-{Latex[length=2.35mm,width=1.75mm]}, line width=.65pt, draw=proofink!72},
  dasharr/.style={-{Latex[length=2.35mm,width=1.75mm]}, line width=.70pt, dashed, draw=proofpurple!85!black},
  lab/.style={font=\scriptsize, fill=white, inner xsep=3pt, inner ysep=1pt, text=proofmuted},
  glabel/.style={font=\scriptsize\scshape, text=proofmuted, fill=white, inner xsep=4pt, inner ysep=1pt}
]

\node[need] (need) at (-7.25,0) {\textbf{Target identity: the hard step}\\[-1pt]
For every $h\in C_c^\infty(\mathbb R^d)$ and $\lambda>0$,
\[
  \bar\pi(h)=\lambda\,\bar\pi(R_\lambda h).
\]
This is the resolvent form of stationarity.};

\node[bridge] (lap) at (0,0) {\textbf{Laplace uniqueness converts resolvents to invariance}\\\vspace{0.1in}
Define the defect $A_h(t):=\bar\pi(P_t h)-\bar\pi(h).$
Using $R_\lambda h=\int_0^\infty e^{-\lambda t}P_t h\,dt$, Fubini turns the target identity into
\[
\int_0^\infty e^{-\lambda t}A_h(t)\,dt=0\qquad(\lambda>0).
\]};

\node[inv] (inv) at (7.25,0) {\textbf{Semigroup invariance}\\[-1pt]
Thus $\bar\pi(P_t h)=\bar\pi(h)$ for all $t\ge0$.\\[-1pt]
Density of $C_c^\infty|_E$ in $C_0(E)$, plus contraction of $P_t$, extends this to all $\phi\in C_0(E)$. Hence $\bar\pi P_t=\bar\pi$.};

\node[why] (jordan) at (3.60,-3.90) {\textbf{Interior measure is then forced}\\[-1pt]
For a finite signed invariant measure, positivity gives $|\mu P_t|\le |\mu|P_t$. Total mass equality makes the Jordan parts invariant.\\[-1pt]
Uniqueness of $\pi_0$ gives $\bar\pi=c\pi_0$.};

\node[why] (boundary) at (-3.60,-3.90) {\textbf{Boundary measures then follow}\\[-1pt]
Subtract $c$ times the stationary BAR. On each stratum $S_A$, $R_{AA}^{-T}$ prescribes the active oblique jets $(D_i f)_{i\in A}$.\\[-1pt]
Induction over $|A|$ gives $\bar\nu_i=c\nu_i^0$.};

\node[finalbox] (final) at (0,-6.15) {\textbf{Signed BAR uniqueness}\\[-1pt]
$\displaystyle (\bar\pi,\bar\nu_1,\ldots,\bar\nu_d)=c(\pi_0,\nu_1^0,\ldots,\nu_d^0).$};

\draw[arr] (need.east) -- (lap.west);
\draw[arr] (lap.east) -- (inv.west);
\coordinate (invdown) at ($(inv.south)+(0,-0.95)$);
\coordinate (jup) at (jordan.north |- invdown);
\draw[arr] (inv.south) -- (invdown) -- (jup) -- (jordan.north);
\draw[arr] (jordan.west) -- node[lab,above] {$\bar\pi=c\pi_0$} (boundary.east);
\draw[arr] (boundary.south) to[out=-90,in=120] (final.north west);
\draw[arr] (jordan.south) to[out=-90,in=60] (final.north east);

\node[bad] (obstacle) at (-6.45,-9.25) {\textbf{Why direct insertion fails}\\[-1pt]
The natural test is $g=R_\lambda h$, because $(\lambda-L)g=h$ in the interior.\\[-1pt]
But the BAR accepts bounded $C^2$ tests and boundary terms $D_i f$. At corners, $g$ need not have a classical ambient gradient, so $D_i g=0$ is unavailable.};

\node[good] (projected) at (0,-9.25) {\textbf{Projected derivative replaces a boundary gradient}\\[-1pt]
For feasible inward directions,
\[
  \partial_w^+g(x)=\Lambda_x(\mathsf L_xw).
\]
If $x\in F_i$, then $\mathsf L_xR_i=0$, hence the ambient linear extension satisfies $\ell_x(R_i)=0$. This is the usable oblique information.};

\node[good] (smooth) at (6.45,-9.25) {\textbf{One-sided smoothing turns it into BAR tests}\\[-1pt]
Smooth only from inside:
\[
  g_\varepsilon(x)=\int \rho(w)g(x+\varepsilon w)\,dw .
\]
The support $\operatorname{supp}\rho\Subset(0,\infty)^d$ keeps every sampled direction feasible. Integration by parts and domination give $D_i g_\varepsilon\to0$ on $F_i$.};

\node[action] (barlimit) at (0,-12.55) {\textbf{Measure--Neumann approximation}\\[-1pt]
Use $g_\varepsilon\in C_b^2(E)$ directly. Apply the signed BAR and send $\varepsilon\downarrow0$. Boundary integrals vanish for every finite signed $\bar\nu_i$; interior terms converge to $\bar\pi(h)=\lambda\bar\pi(R_\lambda h)$.};

\draw[arr] (obstacle.east) -- (projected.west);
\draw[arr] (projected.east) -- (smooth.west);
\draw[arr] (obstacle.south) to[out=-90,in=160] (barlimit.west);
\draw[arr] (projected.south) -- (barlimit.north);
\draw[arr] (smooth.south) to[out=-90,in=20] (barlimit.east);

\coordinate (feedbackleft) at (-9.65,-12.55);
\draw[dasharr] (barlimit.west) -- (feedbackleft) |- (need.west);

\begin{scope}[on background layer]
  \node[draw=proofblue!13, fill=proofblue!3, rounded corners=5pt, inner sep=10pt, fit=(need)(lap)(inv)(jordan)(boundary)(final)] (macrogrp) {};
  \node[glabel, anchor=south west] at ($(macrogrp.north west)+(0.25,-0.02)$) {$\bar\pi(h)=\lambda\bar\pi(R_\lambda h)$ leads to signed BAR uniqueness};
  \node[draw=proofgreen!16, fill=proofgreen!4, rounded corners=5pt, inner sep=10pt, fit=(obstacle)(projected)(smooth)(barlimit)] (analyticgrp) {};
  \node[glabel, anchor=south west] at ($(analyticgrp.north west)+(0.25,-0.02)$) {Proving $\bar\pi(h)=\lambda\bar\pi(R_\lambda h)$: how the missing resolvent boundary regularity is bypassed};
\end{scope}

\end{tikzpicture}%
}
\caption{\textbf{Proof architecture for the uniqueness of the Harrison-Reiman class.} The lower half is the analytic insertion mechanism: \cref{prop:resolvent-regularity-projected} supplies the projected derivative used by the one-sided smoothing, and \cref{prop:measure-Neumann} turns the smoothed functions into admissible BAR tests.  The upper half is the soft reduction: the resulting resolvent identity gives semigroup invariance, then signed uniqueness of the interior measure and finally the boundary measures.}
\label{fig:proof-architecture}
\end{figure}
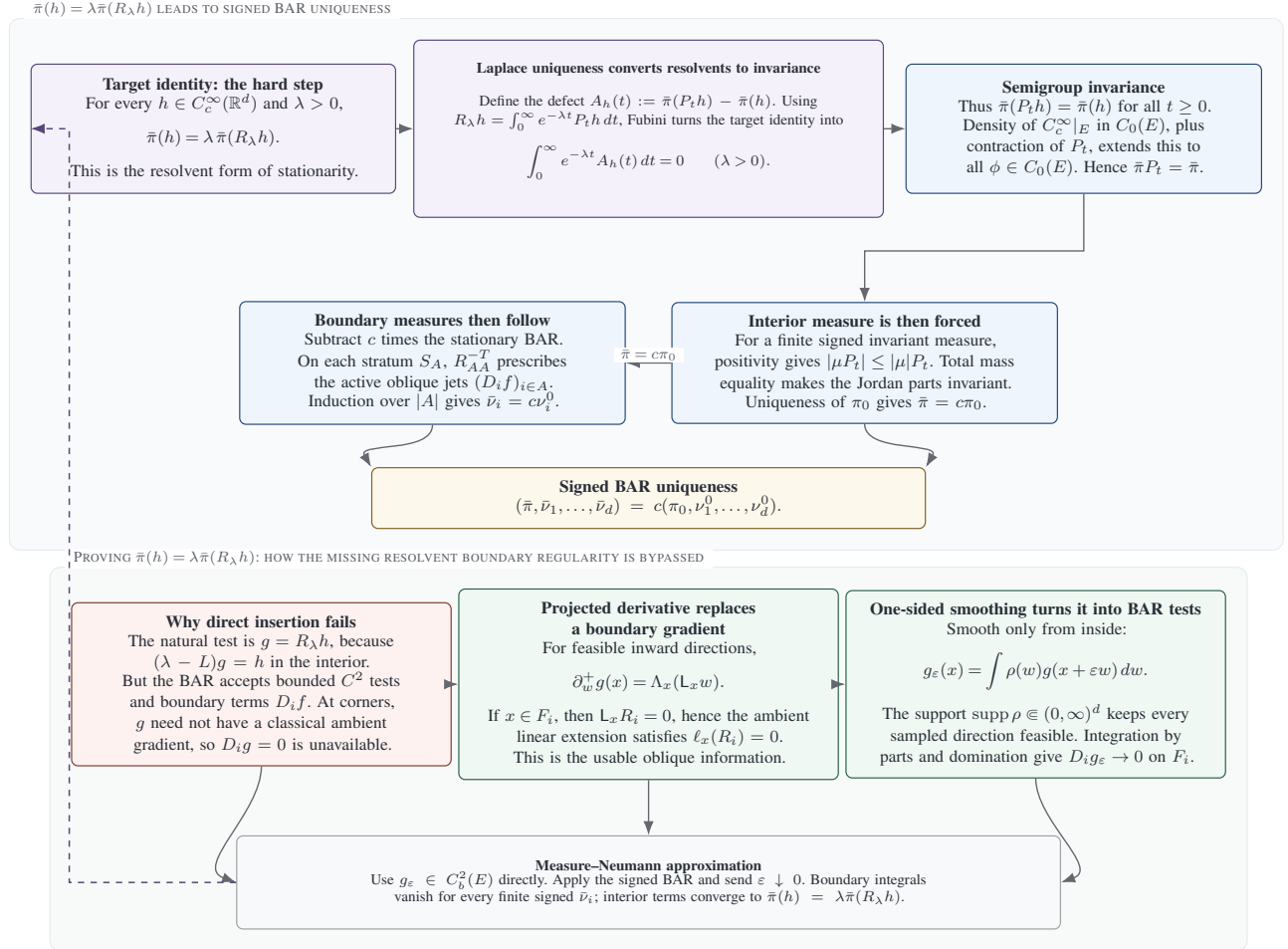 

The diagram also shows why the proof first studies the nonsmooth resolvent before carrying out the smoothing.  For fixed $h$ and $\lambda$, let $g=R_\lambda h$.  The smoothed BAR tests used later are
\[
        g_\varepsilon(x)=\int \rho(w)g(x+\varepsilon w)\,dw,
        \qquad \supp\rho\subset(1,2)^d.
\]
They must approximate $g$ in the interior equation while also satisfying an asymptotic oblique-Neumann condition on each face: $D_i g_\varepsilon\to0$ in pairings with arbitrary finite signed measures on $F_i$.  This is why \cref{sec:resolvent-regularity} proves a boundary statement for $g$ itself before any smoothing is introduced.  Although $g$ need not be $C^2$ on the closed orthant, its feasible one-sided derivatives exist at boundary points and factor through the active tangent projection; the resulting linear extension $\ell_x$ satisfies $\ell_x(R_i)=0$ on active faces, acting as an analog to the classical gradient.  The one-sided convolution $g_\varepsilon$ in \cref{sec:smoothing} is then precisely designed to inherit this first-order oblique flatness in the weaker, measure-level form needed by the BAR.

\section{Resolvent regularity and projected boundary derivatives}\label{sec:resolvent-regularity}

The goal of this section is to prove \cref{prop:resolvent-regularity-projected}, the input that makes the measure--Neumann approximation in \cref{sec:smoothing} possible.  Section~\ref{sec:smoothing} will construct $g_\varepsilon\in C_b^2(E)$ from $g=R_\lambda h$ and will need three properties: $g_\varepsilon\to g$, $(\lambda-L)g_\varepsilon\to h$, and $D_i g_\varepsilon\to0$ on $F_i$ after integration against arbitrary finite signed boundary measures.  The first two properties come from interior smoothing and the interior resolvent equation.  The third property comes from the boundary information proved here: at a boundary point, the feasible directional derivative of $g$ factors through the active tangent projection.  Combining this factorization with the identity $\mathsf L_xR_i=0$ gives the usable oblique information $\ell_x(R_i)=0$ on $F_i$, which is exactly what later forces $D_i g_\varepsilon\to0$ in boundary-measure pairings.

The proof has two ingredients.  The algebraic ingredient is the nonsingularity of every active principal reflection block, which gives the explicit projection $\mathsf L_x$.  The stochastic ingredient is external: the Lipshutz--Ramanan initial-condition derivative theorem for the normalized Harrison--Reiman reflected diffusion, together with well posedness, strong-continuity property, and the synchronous Lipschitz estimate.  The source-to-assumption conversion is carried out in the proof of \cref{thm:regularity-SRBM}, where each source hypothesis is recalled in the present orthant specialization and verified with a self-contained argument; no unlisted regularity or boundary conclusion is used.

For $x\in E$, define $I(x)=\{i\in J:x_i=0\},$ and put
\begin{align}
 G_x=\{w\in\R^d:w_i\ge0\text{ for }i\in I(x)\}, \qquad
 H_x=\{v\in\R^d:v_i=0\text{ for }i\in I(x)\}.
\end{align}

\begin{proposition}[Resolvent regularity and projected boundary derivatives]\label{prop:resolvent-regularity-projected}
Let $\lambda>0$ and let $h\in C_c^\infty(\mathbb R^d)$ be regarded as a function on $E$. Define $g(x)=R_\lambda h(x)
 :=\mathbb E\left[\int_0^\infty e^{-\lambda t}h(Z_t^x)\,dt\right],$ then we have:
\begin{enumerate}[label=\textup{(\roman*)}]
\item $g$ is bounded and globally Lipschitz on $E$.
\item $g$ is a classical solution of the resolvent equation in $E^\circ$; more precisely, $g\in C^\infty(E^\circ)$ and $ (\lambda-L)g=h$ in $E^\circ$.
\item At each $x\in E$, feasible one-sided directional derivatives $\partial_w^+g(x)$ exist for $w\in G_x$.
\item If $A=I(x)$, then the principal-block projection
\(
 \mathsf L_xv=v-R_A R_{AA}^{-1}v_A
\)
maps $\mathbb R^d$ onto $H_x$, and there is a linear functional $\Lambda_x:H_x\to\mathbb R$ such that
\[
 \partial_w^+g(x)=\Lambda_x(\mathsf L_xw),
 \qquad w\in G_x.
\]
\item With $\ell_x(v)=\Lambda_x(\mathsf L_xv)$, we have  $\ell_x(R_i)=0$, for $i\in I(x).$
\end{enumerate}
\end{proposition}

\subsection{Matrix normalization and active-set projections}

Normalize the reflection directions by
\[
 \Delta=\diag(R_{11},\ldots,R_{dd}),
 \qquad
 \widehat R=R\Delta^{-1},
 \qquad
 d_i=\widehat R_i=R_i/R_{ii}.
\]
Positive rescaling of a reflection direction only rescales its regulator and does not change the reflected path.

\begin{lemma}[Principal block projection]\label{lem:Mmatrix}\label{lem:projection}
The normalized matrix has the Harrison--Reiman form
\begin{equation}\label{eq:HR-form}
 \widehat R=I-P^T,
 \qquad P\ge0,
 \qquad \rho(P)<1.
\end{equation}
Every principal submatrix $R_{AA}$ is a nonsingular $M$-matrix and $R_{AA}^{-1}\ge0$. In particular, for every nonempty $A\subset J$, the active directions $\{d_i:i\in A\}$ are linearly independent. For $A\subset J$, define
\begin{equation}\label{eq:projection-formula}
 L_Av=v-R_A R_{AA}^{-1}v_A,
\end{equation}
with $L_\varnothing$ equal to the identity. If $A=I(x)$, then $L_A=\mathsf L_x$ is the (unique) linear map from $\mathbb R^d$ to $H_x$ such that $\mathsf L_xv-v\in\operatorname{span}\{R_i:i\in A\}$. Moreover,
\begin{equation}\label{eq:projection-kills}
 L_A R_i=0,
 \qquad i\in A,
\end{equation}
and
\begin{equation}\label{eq:projection-uniform-constant}
 C_{\mathsf L}:=\max_{A\subset J}\left\|I-R_A R_{AA}^{-1}\pi_A\right\|<\infty,
\end{equation}
where $\pi_Av=v_A$ and the expression for $A=\varnothing$ is the identity.
\end{lemma}

\begin{proof}
The diagonal of $\widehat R$ is one and its off-diagonal entries are nonpositive, so $P^T:=I-\widehat R$ is nonnegative. Also
\[
 \widehat R^{-1}=\Delta R^{-1}\ge0.
\]
By Perron--Frobenius, in the standard nonnegative-matrix form summarized for example in \cite[Chapter~2]{BermanPlemmons1994}, $P^T$ has a nonzero vector $v\ge0$ with $P^Tv=\rho(P)v$. If $\rho(P)=1$, then $\widehat Rv=0$, contradicting invertibility. If $\rho(P)>1$, then $\widehat Rv=(1-\rho(P))v\le0$; multiplying by $\widehat R^{-1}\ge0$ gives $v\le0$, again a contradiction. Hence $\rho(P)<1$.

For a principal index set $A$, the principal block $(P^T)_{AA}$ is nonnegative and $\rho((P^T)_{AA})\le\rho(P^T)<1$. One direct verification of the inequality is that $((P^T)_{AA})^n$ is entrywise bounded by the $AA$ block of $(P^T)^n$, after which Gelfand's formula applies. Therefore
\[
 (I_A-(P^T)_{AA})^{-1}=\sum_{n=0}^\infty ((P^T)_{AA})^n\ge0.
\]
Since $R_{AA}=(I_A-(P^T)_{AA})\Delta_A$, it follows that
\[
 R_{AA}^{-1}=\Delta_A^{-1}(I_A-(P^T)_{AA})^{-1}\ge0.
\]
Linear independence of the active normalized columns follows by restricting a relation $\sum_{i\in A}a_i d_i=0$ to rows in $A$. The maximum in \eqref{eq:projection-uniform-constant} is finite because the active-set lattice is finite.
For the projection claim, a vector of the form $v-R_Aa$ belongs to $H_x$ exactly when $v_A-R_{AA}a=0$. The preceding paragraph gives $a=R_{AA}^{-1}v_A$. If $v=R_i$ with $i\in A$, then $v_A=R_{AA}e_i$, which proves \eqref{eq:projection-kills}.
\end{proof}

\subsection{Regularity of SRBM}\label{sec:regularity-SRBM}

This subsection proves \cref{thm:regularity-SRBM}, the stochastic regularity
statement used in \cref{prop:resolvent-regularity-projected}.  More specifically, for
$g=R_\lambda h$, we will need the reflected semigroup on $C_0(E)$, a
synchronous Lipschitz estimate for paths driven by the same Brownian motion, and
a pathwise derivative with respect to the initial condition.  The derivative
statement is the key boundary input: at a boundary point, the initial
perturbation is projected onto the active tangent space, and the active
reflection directions are killed by this projection.  This is the stochastic
origin of the oblique flatness used in the one-sided smoothing argument.

These properties follow from the reflected-diffusion results in
\cite{LipshutzRamanan2019,LipshutzRamanan2021}.  Those results are formulated for
simple polyhedral domains with normalized reflection directions.  Our SRBM is
the constant-coefficient orthant case of that framework, after a harmless
normalization of the reflection directions.  Set
\[
        \Delta=\diag(R_{11},\ldots,R_{dd}),\qquad
        \widehat R=R\Delta^{-1},\qquad
        d_i=\widehat R_i=\frac{R_i}{R_{ii}} .
\]
Then $\langle d_i,e_i\rangle=1$.  Replacing $R_i$ by the positive multiple
$d_i=R_i/R_{ii}$ only rescales the $i$th regulator coordinate and leaves the
reflected path unchanged.  Therefore pathwise statements proved for the
normalized matrix $\widehat R$ apply to the original BAR normalization $R$.

\Cref{thm:regularity-SRBM} records the regularity consequences needed for the
proof of \cref{prop:resolvent-regularity-projected}.  Its proof first places the
present SRBM into the notation of \cite{LipshutzRamanan2019,LipshutzRamanan2021},
then verifies the relevant hypotheses under
\eqref{eq:Mmatrix}--\eqref{eq:stability}, and finally applies the corresponding
existence, Lipschitz, and derivative results of \cite{LipshutzRamanan2019,LipshutzRamanan2021}.

\begin{theorem}[Regularity of SRBM]\label{thm:regularity-SRBM}
Under the standing assumptions \eqref{eq:Mmatrix}--\eqref{eq:stability}, the
following hold.
\begin{enumerate}[label=\textup{(\roman*)}]
\item For each $x\in E$ and each prescribed Brownian motion there is a pathwise
unique SRBM $Z^x$, and $Z^x$ is strong Markov.

\item The semigroup $(P_t)$ maps $C_0(E)$ into itself and is strongly continuous
there.

\item There is a constant $K_\Gamma<\infty$, depending only on the normalized
reflection data, such that synchronous solutions satisfy, for all $x,y\in E$ and
$t\ge0$,
\begin{equation}\label{eq:pathwise-Lipschitz}
        \sup_{0\le s\le t}\abs{Z_s^x-Z_s^y}
        \le K_\Gamma\abs{x-y}
        \qquad\text{almost surely.}
\end{equation}

\item For every $x\in E$ there is an adapted RCLL derivative process
$\mathsf J_t^x\in\operatorname{Lin}(H_x,\R^d)$, $t\ge0$.  For each fixed
$w\in G_x$, on an event of probability one the derivative
\[
        \partial_w Z_t^x
        :=
        \lim_{\varepsilon\downarrow0}
        \frac{Z_t^{x+\varepsilon w}-Z_t^x}{\varepsilon}
\]
exists for every $t\ge0$.  Moreover, for every fixed $t>0$,
\begin{equation}\label{eq:fixed-time-derivative}
        \partial_w Z_t^x
        =
        \mathsf J_t^x[\mathsf L_xw]
        \qquad\text{almost surely.}
\end{equation}

\item For every $x\in E$ and every fixed $u\in H_x$,
\[
        \abs{\mathsf J_t^x[u]}
        \le
        K_\Gamma\abs{u}
        \qquad
        \text{for }dt\otimes\mathbb P\text{-almost every }(t,\omega).
\]
\end{enumerate}
\end{theorem}

To prove \cref{thm:regularity-SRBM}, we use the following results for reflected
diffusions in simple polyhedra \cite{LipshutzRamanan2019,LipshutzRamanan2021}.  The general framework of \cite{LipshutzRamanan2019,LipshutzRamanan2021} is a more flexible version
of the same reflected-diffusion equation: it allows a simple polyhedral domain,
normalized face directions, and parameter-dependent coefficients.  Our orthant
SRBM is obtained from that framework by taking constant coefficients and
normalized columns $d_i=R_i/R_{ii}$; the only difference from the BAR notation is
the harmless positive rescaling of the regulator coordinates.

Recall that \cite{LipshutzRamanan2019,LipshutzRamanan2021} use the following notation for the general SRBM framework, where we specialized the notation to the spatially homogeneous case.  Let parameters $\alpha\in U$, where the parameter family
$U$ is open, and let
\[
        G=\bigcap_{i\in J}\{x:\langle x,n_i\rangle\ge c_i\},
        \qquad J=\{1,\ldots,d\},
\]
be a minimally represented simple polyhedron with unit inward normals $n_i$,
faces $F_i=\{x\in G:\langle x,n_i\rangle=c_i\}$, and active set
$I_G(x)=\{i:x\in F_i\}$.  The normalized reflection directions satisfy
$\langle d_i(\alpha),n_i\rangle=1$.  In the spatially-homogeneous-coefficient specialization we care about, the family of reflected diffusions parameterized by $\alpha$ is written as
\[
        Z_t^{\alpha,x}
        =
        x+b(\alpha)t+\sigma(\alpha)W_t
        +\sum_{i\in J}d_i(\alpha)Y_i^{\alpha,x}(t),
\]
where each $Y_i^{\alpha,x}$ is continuous, nondecreasing, starts from zero, and
increases only when $Z^{\alpha,x}\in F_i$.  Put
$a(\alpha)=\sigma(\alpha)\sigma(\alpha)^T$,
$N=(n_1,\ldots,n_d)$, and
$\mathcal D(\alpha)=(d_1(\alpha),\ldots,d_d(\alpha))$.  For $x\in G$, define
$C_G(x)=\{w:\langle w,n_i\rangle\ge0,\ i\in I_G(x)\}$ and
$H_G(x)=\{v:\langle v,n_i\rangle=0,\ i\in I_G(x)\}$.  In the orthant
specialization, $G=E$, $n_i=e_i$, $C_G(x)=G_x$, and $H_G(x)=H_x$. Then, \cite{LipshutzRamanan2019, LipshutzRamanan2021} gives the following proposition.

\begin{proposition}[Reflected diffusions in simple polyhedra]\label{prop:LR-input}
In the setting just described, fix $\alpha\in U$.  Assume the following
hypotheses.
\begin{enumerate}[label=\textup{(A\arabic*)},leftmargin=2.1em]
\item $G$ is minimally represented and simple, $U$ is open, and
$\alpha\mapsto d_i(\alpha)$, $b(\alpha)$, and $\sigma(\alpha)$ are $C^1$ with
bounded first derivatives and local H\"older regularity.

\item $a(\alpha)$ is uniformly elliptic: $v^Ta(\alpha)v\ge\theta|v|^2$ for some
$\theta>0$ and all $v\in\R^d$.

\item $N^T\mathcal D(\alpha)$ is a nonsingular $M$-matrix.

\item The reflection matrix is constant in the parameter, or more generally
$\partial_\alpha\mathcal D(\alpha)$ is bounded.
\end{enumerate}

Then the following conclusions are available under the assumptions indicated.

\begin{enumerate}[label=\textup{(C\arabic*)},leftmargin=2.1em]
\item \emph{(Well posedness; uses \textup{(A1)} and \textup{(A3)},
\cite[Theorem~2.8]{LipshutzRamanan2021}.)}
For each $x\in G$ and each prescribed Brownian motion, there is a pathwise
unique reflected diffusion $Z^{\alpha,x}$, and it is strong Markov.

\item \emph{(Lipschitz extended Skorokhod map; uses \textup{(A1)} and
\textup{(A3)}, \cite[Proposition~2.6]{LipshutzRamanan2019}.)}
The extended Skorokhod problem associated with $(G,d_i(\alpha))$ is well posed,
and its extended Skorokhod map $\bar\Gamma^\alpha$ is Lipschitz on compact time
intervals: $\sup_{s\le t}\abs{\bar\Gamma^\alpha(f)(s)-\bar\Gamma^\alpha(g)(s)}
\le K_\Gamma\sup_{s\le t}\abs{f(s)-g(s)}$ for some $K_\Gamma<\infty$, all
continuous inputs $f,g$, and all $t\ge0$.

\item \emph{(Boundary jitter; uses \textup{(A2)} in the above setting,
\cite[Theorem~3.3]{LipshutzRamanan2019}.)}
Uniform ellipticity implies the boundary jitter property required for the
pathwise derivative theorem.

\item \emph{(Derivative projection; uses \textup{(A1)} and \textup{(A3)},
\cite[Lemma~3.11]{LipshutzRamanan2019}.)}
For each $x\in G$ there is a unique linear projection
$\mathcal L_x^\alpha:\R^d\to H_G(x)$ such that
$\mathcal L_x^\alpha v-v\in
\operatorname{span}\{d_i(\alpha):i\in I_G(x)\}$ for every $v\in\R^d$.

\item \emph{(Pathwise differentiability; uses \textup{(A1)}--\textup{(A4)} and
\textup{(C3)}--\textup{(C4)}, \cite[Theorem~3.13 and Corollary~3.15]{LipshutzRamanan2019}.)}
For this theorem, note that \textup{(A3)} implies Condition~2.10 of
\cite{LipshutzRamanan2019}: if $A\subset J$ and
$\sum_{i\in A}c_i d_i(\alpha)=0$, then
$(N^T\mathcal D(\alpha))_{AA}c_A=0$, and every principal submatrix of a
nonsingular $M$-matrix is nonsingular, so $c_A=0$.  Hence
\cite[Lemma~3.9]{LipshutzRamanan2019} gives the exceptional set of this theorem $\mathcal W^\alpha=\varnothing$.  Therefore,
for each $x\in G\backslash \mathcal W^\alpha=G$ there is an adapted RCLL derivative process
$J_t^{\alpha,x}\in\operatorname{Lin}(H_G(x),\R^d)$.  For every fixed
$w\in C_G(x)$, we have almost surely,
$\partial_w Z_t^{\alpha,x}:=\lim_{\varepsilon\downarrow0}
\varepsilon^{-1}(Z_t^{\alpha,x+\varepsilon w}-Z_t^{\alpha,x})$ exists for every
$t\ge0$, is continuous at every
$t>0$ such that $Z_t^{\alpha,x}\in G^\circ$, and its right-continuous regularization satisfies
$\lim_{s\downarrow t}\partial_w Z_s^{\alpha,x}
=J_t^{\alpha,x}[\mathcal L_x^\alpha w]$ for all $t\ge0$.

\item \emph{(Fixed-time interior statement; uses \textup{(A1)}--\textup{(A4)}
and \textup{(C3)}--\textup{(C4)}, \cite[Lemma~4.13]{LipshutzRamanan2019}.)}
For every fixed $t>0$, $\mathbb P(Z_t^{\alpha,x}\in G^\circ)=1$.
\end{enumerate}
\end{proposition}

\begin{comment}
    \begin{proof}[Proof of Proposition \ref{prop:LR-input} in \cite{LipshutzRamanan2019,LipshutzRamanan2021}]
The proposition restates the cited results in the notation above.  The
$M$-matrix sufficient condition in \textup{(A3)} is
\cite[Lemma~2.6 and Remark~2.7]{LipshutzRamanan2021}.  Conclusion
\textup{(C1)} is
\cite[Theorem~2.8]{LipshutzRamanan2021}.  The Lipschitz estimate
\textup{(C2)} is \cite[Proposition~2.6]{LipshutzRamanan2019}.  The boundary
jitter input and empty-exceptional-set criterion in \textup{(C3)} are
\cite[Theorem~3.3 and Lemma~3.9]{LipshutzRamanan2019}.  The projection in
\textup{(C4)} is \cite[Lemma~3.11]{LipshutzRamanan2019}.  The derivative
statement \textup{(C5)} is
\cite[Theorem~3.13 and Corollary~3.15]{LipshutzRamanan2019}.  The fixed-time
interior statement \textup{(C6)} is
\cite[Lemma~4.13]{LipshutzRamanan2019}.
\end{proof}

\end{comment}
\begin{proof}[Proof of \cref{thm:regularity-SRBM}]
We apply \cref{prop:LR-input} to the constant parameter family
\[
        G=E,\quad
        n_i=e_i,\quad
        c_i=0,\quad
        d_i(\alpha)=d_i=\frac{R_i}{R_{ii}},\quad
        b(\alpha)=\mu,\quad
        \sigma(\alpha)=\Sigma^{1/2},
        \quad \alpha\in U=(-1,1).
\]
Thus $\mathcal D(\alpha)=\widehat R$ and $N=I$. Here constant parameter family means that the domain, reflection directions, drift, and dispersion do
not depend on $\alpha$.  Now we verify that the Harrison-Reiman Class satisfies all assumptions (A1)-(A4).

\smallskip
\noindent\emph{Verification of \textup{(A1)}.}
The orthant is the minimally represented simple cone
$E=\bigcap_i\{x:\langle x,e_i\rangle\ge0\}$.  Simplicity follows because the
coordinate normals are linearly independent on every active set.  Minimality
follows because, if the $i$th half-space is removed, then the point $-e_i$
satisfies all remaining half-space inequalities but does not belong to $E$.  The
parameter set $U=(-1,1)$ is open.  The maps $d_i(\alpha)$, $b(\alpha)$, and
$\sigma(\alpha)$ are constant, hence $C^1$; all first derivatives are zero, and
therefore bounded and locally H\"older.  The normalization
$\langle d_i,e_i\rangle=1$ holds by definition.  This verifies \textup{(A1)}.

\smallskip
\noindent\emph{Verification of \textup{(A2)}.}
Here $a(\alpha)=\Sigma$.  Since $\Sigma$ is symmetric positive definite,
\[
        v^T\Sigma v
        \ge
        \lambda_{\min}(\Sigma)|v|^2,
        \qquad v\in\R^d .
\]
Thus \textup{(A2)} holds with $\theta=\lambda_{\min}(\Sigma)>0$.

\smallskip
\smallskip
\noindent\emph{Verification of \textup{(A3)}.}
Here $N=I$, $\mathcal D(\alpha)=\widehat R=I-P^T$, and hence
$N^T\mathcal D(\alpha)=\widehat R$.  By \cref{lem:Mmatrix}, $\widehat R$ is a
nonsingular $M$-matrix.  This verifies \textup{(A3)}.

\smallskip
\noindent\emph{Verification of \textup{(A4)}.}
The reflection matrix $\mathcal D(\alpha)=\widehat R$ is constant in $\alpha$, so
$\partial_\alpha\mathcal D(\alpha)=0$.  This verifies \textup{(A4)}.

All assumptions \textup{(A1)}--\textup{(A4)} of \cref{prop:LR-input} have now
been verified.

Conclusion \textup{(C1)} gives pathwise existence, uniqueness, and the strong
Markov property for the normalized reflected diffusion.  Since
$R_i=R_{ii}d_i$ with $R_{ii}>0$, replacing the normalized local time by the
correspondingly rescaled regulator leaves the reflected path unchanged.  Hence
the same pathwise existence, uniqueness, and strong Markov conclusions also hold
for the original BAR normalization $R$.  This proves assertion \textup{(i)}.

For assertion \textup{(iii)}, let the two processes start from $x$ and $y$ and
be driven by the same Brownian path.  Their free inputs are
$f_x(s)=x+\mu s+\Sigma^{1/2}W_s$ and
$f_y(s)=y+\mu s+\Sigma^{1/2}W_s$, so
$\sup_{s\le t}|f_x(s)-f_y(s)|=|x-y|$.  Applying conclusion \textup{(C2)} gives
\eqref{eq:pathwise-Lipschitz}, proving assertion \textup{(iii)}.

Assertion \textup{(ii)} follows from \eqref{eq:pathwise-Lipschitz} and Brownian
continuity.  Put $X_t=\mu t+\Sigma^{1/2}W_t$.  Comparing the input $x+X$ with
the constant input $x$ gives
\begin{equation}\label{eq:increment-bound}
        \sup_{s\le t}\abs{Z_s^x-x}
        \le
        K_\Gamma\sup_{s\le t}\abs{X_s}
        \qquad\text{almost surely.}
\end{equation}
If $h\in C_0(E)$, then $h$ is uniformly continuous.  Hence
\eqref{eq:pathwise-Lipschitz} gives continuity of $x\mapsto P_th(x)$.  If $h$ is
supported in the ball of radius $r$, then
\[
        |P_th(x)|
        \le
        \|h\|_\infty
        \mathbb P\!\left(
             K_\Gamma\sup_{s\le t}|X_s|\ge |x|-r
        \right),
\]
which tends to zero as $|x|\to\infty$; approximation by compactly supported
functions gives $P_tC_0(E)\subset C_0(E)$.  Finally, if $\omega_h$ is the modulus
of continuity of $h$, then \eqref{eq:increment-bound} gives
\[
        \sup_{x\in E}|P_th(x)-h(x)|
        \le
        \mathbb E\,
        \omega_h\!\left(K_\Gamma\sup_{s\le t}|X_s|\right)
        \longrightarrow0
        \qquad (t\downarrow0)
\]
by bounded convergence.  Thus $(P_t)$ is strongly continuous on $C_0(E)$.  This
proves assertion \textup{(ii)}.

We now prove assertions \textup{(iv)}--\textup{(v)}.  By conclusion
\textup{(C4)}, the derivative projection at $x$ is the unique linear map onto
$H_x$ whose difference from the identity lies in
$\operatorname{span}\{d_i:i\in I(x)\}$.  Since
$\operatorname{span}\{d_i:i\in A\}=\operatorname{span}\{R_i:i\in A\}$ for every
$A\subset J$, uniqueness and \cref{lem:projection} identify this projection with
$\mathsf L_x$.

Conclusion \textup{(C5)}, applied to the constant parameter family and with
parameter direction equal to zero, gives the derivative process $\mathsf J_t^x$
and the directional derivative $\partial_w Z_t^x$ for every fixed $w\in G_x$.
It also gives continuity of $s\mapsto\partial_wZ_s^x$ at every $t>0$ such that
$Z_t^x\in E^\circ$, together with the projected right-continuous
regularization.

Conclusion \textup{(C6)} gives $\mathbb P(Z_t^x\in E^\circ)=1$ for every fixed
$t>0$.  Therefore, for every fixed $w\in G_x$ and $t>0$, on an event of
probability one,
\[
        \partial_w Z_t^x
        =
        \lim_{s\downarrow t}\partial_w Z_s^x
        =
        \mathsf J_t^x[\mathsf L_xw].
\]
This proves assertion \textup{(iv)}.

For assertion \textup{(v)}, fix $u\in H_x$.  Then $u\in G_x$ and
$\mathsf L_xu=u$.  For small $\varepsilon>0$, $x+\varepsilon u\in E$.  Applying
\eqref{eq:pathwise-Lipschitz} with $y=x+\varepsilon u$, dividing by
$\varepsilon$, and letting $\varepsilon\downarrow0$ gives
$\abs{\partial_u Z_t^x}\le K_\Gamma\abs{u}$ on the event where the directional
derivative exists for all $t\ge0$.  Combining this bound with
\eqref{eq:fixed-time-derivative} gives
$\abs{\mathsf J_t^x[u]}\le K_\Gamma\abs{u}$ for every fixed $t>0$, almost surely.
Since $\mathsf J^x$ is RCLL, Fubini gives the same bound for
$dt\otimes\mathbb P$-almost every $(t,\omega)$.  This proves assertion
\textup{(v)} and completes the proof.

\end{proof}

\subsection{The probabilistic resolvent and its boundary directional derivative}\label{sec:resolvent}

The purpose of this subsection is to convert the pathwise derivative package into a boundary identity for the probabilistic resolvent. We prove only that the resolvent is a classical solution of the resolvent equation in the interior, then differentiate the time integral in feasible directions. The resulting boundary derivative is an algebraic linear functional; no classical gradient at a corner is assumed.

Fix $\lambda>0$ and $h\in C_c^\infty(\R^d)$. We regard $h$ as a function on $E$ and define
\begin{equation}\label{eq:resolvent}
 g(x)=R_\lambda h(x)
 :=\mathbb E\left[\int_0^\infty e^{-\lambda t}h(Z_t^x)\,dt\right].
\end{equation}

\begin{lemma}[Boundedness and Lipschitz continuity]\label{lem:resolvent-Lipschitz}
The function $g$ is bounded and globally Lipschitz on $E$, with
\begin{equation}\label{eq:resolvent-bounds}
 \norm{g}_\infty\le\frac{\norm{h}_\infty}{\lambda},
 \qquad
 \Lip(g)\le\frac{K_\Gamma\Lip(h)}{\lambda}.
\end{equation}
\end{lemma}

\begin{proof}
The first estimate follows immediately from \eqref{eq:resolvent}. For the second, couple $Z^x$ and $Z^y$ with the same Brownian path. By \eqref{eq:pathwise-Lipschitz},
\[
 \abs{h(Z_t^x)-h(Z_t^y)}
 \le \Lip(h)K_\Gamma\abs{x-y}.
\]
Integrating against $e^{-\lambda t}\,dt$ proves the claim.
\end{proof}

\begin{lemma}[Interior classical solution of the resolvent equation]\label{lem:interior-PDE}
The function $g$ is a classical solution of the resolvent equation in $E^\circ$; more precisely, $g\in C^\infty(E^\circ)$ and
\begin{equation}\label{eq:interior-PDE}
 (\lambda-L)g=h
 \qquad\text{in }E^\circ.
\end{equation}
\end{lemma}

\begin{proof}
Fix concentric balls $B'\Subset B\Subset E^\circ$ and let $\tau_B$ be the first exit time from $B$. Before $\tau_B$, the reflected process is the unconstrained diffusion with generator $L$. The strong Markov property gives, for $x\in B$,
\begin{equation}\label{eq:stopped-resolvent}
 g(x)=\mathbb E_x\left[
 \int_0^{\tau_B}e^{-\lambda t}h(Z_t)\,dt
 +e^{-\lambda\tau_B}g(Z_{\tau_B})
 \right].
\end{equation}
Since \(g\in C(\overline B)\) and \(\partial B\) is compact, Stone--Weierstrass applied to the restrictions of polynomials on \(\mathbb R^d\) to \(\partial B\) gives polynomials \(p_n\) with \(\|p_n-g\|_{L^\infty(\partial B)}\to0\). Setting \(\varphi_n=p_n|_{\overline B}\), we have \(\varphi_n\in C^\infty(\overline B)\subset C^{2,\alpha}(\overline B)\) and \(\varphi_n\to g|_{\partial B}\) uniformly. To solve the interior Dirichlet problems we use the classical Schauder solvability theorem, which we recall in the form used.
\begin{theorem}[{\normalfont\cite[Theorem~6.14]{GilbargTrudinger2001}}]
Let $\Omega\subset\R^d$ be a bounded $C^{2,\alpha}$ domain and let $\mathcal A=a^{ij}D_{ij}+b^iD_i+c$ be strictly elliptic on $\Omega$, that is, $a^{ij}(y)\xi_i\xi_j\ge\theta_0\abs{\xi}^2$ for some $\theta_0>0$ and all $y\in\Omega$, $\xi\in\R^d$, with coefficients $a^{ij},b^i,c\in C^{\alpha}(\overline\Omega)$ and $c\le0$ on $\Omega$. Then for every $f\in C^{\alpha}(\overline\Omega)$ and every $\varphi\in C^{2,\alpha}(\overline\Omega)$ the Dirichlet problem $\mathcal Au=f$ in $\Omega$, $u=\varphi$ on $\partial\Omega$, has a unique solution $u\in C^{2,\alpha}(\overline\Omega)$.
\end{theorem}
We apply this with $\Omega=B$, $\mathcal A=L-\lambda$, so that in coordinates $a^{ij}=\tfrac12\Sigma_{ij}$, $b^i=\mu_i$, $c=-\lambda$, together with $f=-h$ and $\varphi=\phi_n$. The four hypotheses hold in the present setting:
\begin{enumerate}[label=\textup{(\alph*)}]
\item\label{it:gt-domain} $B$ is an open Euclidean ball, hence a $C^\infty$ and a fortiori $C^{2,\alpha}$ domain.
\item\label{it:gt-elliptic} For all $\xi\in\R^d$, $a^{ij}\xi_i\xi_j=\tfrac12\xi^\top\Sigma\xi\ge\tfrac12\lambda_{\min}(\Sigma)\abs{\xi}^2$, and $\lambda_{\min}(\Sigma)>0$ because $\Sigma$ is symmetric positive definite; thus $\mathcal A$ is strictly elliptic with $\theta_0=\tfrac12\lambda_{\min}(\Sigma)$.
\item\label{it:gt-coeff} The coefficients $\tfrac12\Sigma_{ij}$, $\mu_i$, $-\lambda$ are constants, hence lie in $C^{\alpha}(\overline B)$ with vanishing H\"older seminorm; and $c=-\lambda<0\le0$ since $\lambda>0$.
\item\label{it:gt-data} The source $f=-h\in C_c^\infty(\R^d)\subset C^{\alpha}(\overline B)$, and each $\varphi=\phi_n\in C^\infty(\overline B)\subset C^{2,\alpha}(\overline B)$.
\end{enumerate}
Therefore \cite[Theorem~6.14]{GilbargTrudinger2001} yields a unique $u_n\in C^{2,\alpha}(\overline B)$ satisfying
\[
 (\lambda-L)u_n=h\quad\text{in }B,
 \qquad u_n=\phi_n\quad\text{on }\partial B.
\]
Apply It\^o's formula to $e^{-\lambda(t\wedge\tau_B)}u_n(Z_{t\wedge\tau_B})$. Since $u_n$ and its first derivatives are bounded on $\overline B$, the stopped stochastic integral has mean zero. Letting $t\to\infty$ is justified by bounded convergence. Indeed, before $\tau_B$ the process is $x+\mu t+\Sigma^{1/2}W_t$; a nonzero one-dimensional projection is a Brownian motion with drift and exits the bounded projection of $B$ almost surely, so $\tau_B<\infty$ almost surely. This gives the Feynman--Kac representation
\[
 u_n(x)=\mathbb E_x\left[
 \int_0^{\tau_B}e^{-\lambda t}h(Z_t)\,dt
 +e^{-\lambda\tau_B}\phi_n(Z_{\tau_B})
 \right].
\]
Comparing it with \eqref{eq:stopped-resolvent} yields
\[
 \|u_n-g\|_{L^\infty(B)}
 \le \|\phi_n-g\|_{L^\infty(\partial B)}\longrightarrow0.
\]
For $n,m$, the difference $w=u_n-u_m\in C^{2,\alpha}(\overline B)$ solves the homogeneous equation $(\lambda-L)w=h-h=0$ in $B$. To pass to the limit we use the interior Schauder estimate, recalled in the form used.

Next, we use interior Schauder estiamte to prove $u_n$ is Cauchy in $C^{2,\alpha}(B')$:
\begin{theorem}[Interior Schauder estimate {\normalfont \cite[Theorem~6.2]{GilbargTrudinger2001}}]
Let $\mathcal A=a^{ij}D_{ij}+b^iD_i+c$ be strictly elliptic on a domain $\Omega\subset\R^d$ with ellipticity constant $\theta_0>0$ and coefficients bounded in $C^{\alpha}(\Omega)$ by a constant $\Theta$. If $u\in C^{2,\alpha}(\Omega)$ satisfies $\mathcal Au=f$ with $f\in C^{\alpha}(\Omega)$, then for every subdomain $\Omega'\Subset\Omega$,
\[
 \norm{u}_{C^{2,\alpha}(\overline{\Omega'})}
 \le C\bigl(\norm{u}_{L^\infty(\Omega)}+\norm{f}_{C^{\alpha}(\Omega)}\bigr),
 \qquad C=C\bigl(d,\alpha,\theta_0,\Theta,\operatorname{dist}(\Omega',\partial\Omega)\bigr).
\]
\end{theorem}
The operator $\mathcal A=L-\lambda$ satisfies these hypotheses with the ellipticity constant $\theta_0=\tfrac12\lambda_{\min}(\Sigma)$ of \ref{it:gt-elliptic} and the coefficient bound $\Theta=\max\{\tfrac12\norm{\Sigma},\ \abs{\mu},\ \lambda\}$, both independent of $n,m$. Applying it to $w=u_n-u_m$, with $\Omega=B$, $\Omega'=B'$, and $f\equiv0$, gives
\[
 \norm{u_n-u_m}_{C^{2,\alpha}(\overline{B'})}
 \le C_{B',B}\,\norm{u_n-u_m}_{L^\infty(B)},
 \qquad
 C_{B',B}=C\bigl(d,\alpha,\tfrac12\lambda_{\min}(\Sigma),\Theta,\operatorname{dist}(B',\partial B)\bigr),
\]
where the constant $C_{B',B}$ does not depend on $n,m$.
Because $u_n\to g$ uniformly on $B$, the right-hand side tends to zero as $n,m\to\infty$. Thus $(u_n)$ is Cauchy in $C^{2,\alpha}(B')$ and converges there to some $u\in C^{2,\alpha}(B')$. The same sequence converges uniformly to $g$ on $B$, so the $C^{2,\alpha}(B')$ limit must be $u=g$. Passing to the limit in the equations satisfied by the classical solutions $u_n$ gives $g\in C^{2,\alpha}(B')$ and \eqref{eq:interior-PDE} on $B'$. Since the coefficients of $L$ are constant and $h$ is smooth, standard interior elliptic regularity, equivalently the usual bootstrapping by interior estimates, gives $g\in C^\infty(B')$. Since $B'\Subset E^\circ$ was arbitrary, the conclusion follows.
\end{proof}

\begin{proposition}[Directional factorization of the resolvent]\label{prop:resolvent-directional}
Let \(x\in E\). There is a bounded linear functional
\(\Lambda_x:H_x\to \mathbb R\) such that, for every \(w\in G_x\), the one sided
directional derivative exists and satisfies
\begin{equation}
\partial_w^+ g(x)
:=
\lim_{\varepsilon\downarrow0}
\frac{g(x+\varepsilon w)-g(x)}{\varepsilon}
=
\Lambda_x(\mathsf L_xw).
\end{equation}
The linear functional $\ell_x(v):=\Lambda_x(\mathsf L_xv)$ is bounded by
\begin{equation}
|\ell_x(v)|
\le
\frac{K_\Gamma C_{\mathsf L}\|\nabla h\|_\infty}{\lambda}|v|
\end{equation}
for all $v\in \mathbb R^d$. If \(i\in I(x)\), then $\ell_x(R_i)=0.$ Note that no continuity or measurability of the map \(x\mapsto \ell_x\) is asserted.
\end{proposition}

\begin{proof}
Fix \(x\in E\). For \(u\in H_x\), define $\Lambda_x(u)
:=
\mathbb E\int_0^\infty e^{-\lambda t}
\nabla h(Z_t^x)\cdot \mathsf J_t^x[u]\,dt .$ We first record the measurability and integrability facts needed to define
\(\Lambda_x\). Since \(\mathsf J^x\) is adapted and RCLL with values in the finite dimensional space
\(\operatorname{Lin}(H_x,\mathbb R^d)\), for every fixed \(u\in H_x\) the process $(t,\omega)\mapsto \mathsf J_t^x[u](\omega)$ is progressively measurable, hence \(\mathcal B([0,\infty))\otimes\mathcal F\)-measurable.

By Theorem \ref{thm:regularity-SRBM}, for each fixed \(u\in H_x\), we have 
\(
|\mathsf J_t^x[u]|\le K_\Gamma |u|\) holds for \(dt\otimes \mathbb P\)-almost every \((t,\omega)\). Hence  $\mathbb E\int_0^\infty e^{-\lambda t}
\left|\nabla h(Z_t^x)\cdot \mathsf J_t^x[u]\right|\,dt
\le
\frac{K_\Gamma\|\nabla h\|_\infty}{\lambda}|u|<\infty $, \emph{i.e.} the integral
defining \(\Lambda_x(u)\) is absolutely convergent with
respect to \(e^{-\lambda t}dt\otimes\mathbb P\), and
\[
|\Lambda_x(u)|
\le
\frac{K_\Gamma\|\nabla h\|_\infty}{\lambda}|u|,
\qquad u\in H_x .
\]
Since \(\mathsf J_t^x\in \operatorname{Lin}(H_x,\mathbb R^d)\), linearity of
\(\Lambda_x\) follows from linearity of \(\mathsf J_t^x\) and the preceding bound.
Thus \(\Lambda_x\) is a bounded linear functional on \(H_x\).

Now fix \(w\in G_x\). For all sufficiently small \(\varepsilon>0\),
\(x+\varepsilon w\in E\). By the definition of \(g\),
\begin{equation}
\frac{g(x+\varepsilon w)-g(x)}{\varepsilon}
=
\mathbb E\int_0^\infty e^{-\lambda t}
\frac{h(Z_t^{x+\varepsilon w})-h(Z_t^x)}{\varepsilon}\,dt .
\end{equation}
The synchronous Lipschitz estimate gives the deterministic domination
\begin{equation}
\left|
\frac{h(Z_t^{x+\varepsilon w})-h(Z_t^x)}{\varepsilon}
\right|
\le
K_\Gamma \|\nabla h\|_\infty |w| ,
\end{equation}
uniformly for all sufficiently small \(\varepsilon>0\), all \(t\ge0\), and all
sample paths.

On the probability one event in \cref{thm:regularity-SRBM} corresponding to this fixed pair
\((x,w)\), the pathwise directional derivative
\(
\partial_w Z_t^x
=
\lim_{\varepsilon\downarrow0}
\frac{Z_t^{x+\varepsilon w}-Z_t^x}{\varepsilon}\) exists for every \(t\ge0\). Since \(h\in C_c^\infty(\mathbb R^d)\), the
mean value formula gives
\[
\frac{h(Z_t^{x+\varepsilon w})-h(Z_t^x)}{\varepsilon}
\longrightarrow
\nabla h(Z_t^x)\cdot \partial_w Z_t^x
\]
for \(dt\otimes\mathbb P\)-almost every \((t,\omega)\). Dominated convergence
with respect to \(e^{-\lambda t}\,dt\otimes\mathbb P\) therefore yields
\begin{equation}\label{eq:resolvent-derivative-process}
\partial_w^+ g(x)
=
\mathbb E\int_0^\infty e^{-\lambda t}
\nabla h(Z_t^x)\cdot \partial_w Z_t^x\,dt .
\end{equation}

For every fixed \(t>0\), Theorem \ref{thm:regularity-SRBM} gives $\partial_w Z_t^x = \mathsf J_t^x[\mathsf L_xw]$ almost surely. For fixed \(x\) and \(w\), the map $(t,\omega)\mapsto \partial_w Z_t^x(\omega)$ is \(\mathcal B([0,\infty))\otimes\mathcal F\)-measurable, since it is the
pointwise limit of the continuous-in-\(t\) difference quotients $\varepsilon^{-1}\bigl(Z_t^{x+\varepsilon w}-Z_t^x\bigr).$ The process \(t\mapsto \mathsf J_t^x[\mathsf L_xw]\) is measurable because \(\mathsf J^x\) is RCLL. Hence the fixed-time almost sure identity can be integrated in \(t\), giving
the identity \(dt\otimes\mathbb P\)-almost everywhere. By Fubini, this identity holds for \(dt\otimes\mathbb P\)-almost every
\((t,\omega)\). The value at \(t=0\) is irrelevant for the time integral.
Substituting into the preceding display gives
\begin{equation}
\partial_w^+ g(x)
=
\mathbb E\int_0^\infty e^{-\lambda t}
\nabla h(Z_t^x)\cdot \mathsf J_t^x[\mathsf L_xw]\,dt
=
\Lambda_x(\mathsf L_xw),
\end{equation}
which proves the factorization.

Finally, by Lemma \ref{lem:projection}, we have $|\mathsf L_xv|\le C_{\mathsf L}|v|,$ holds for all $v\in\mathbb R^d.$ Therefore
\[
|\ell_x(v)|
=
|\Lambda_x(\mathsf L_xv)|
\le
\frac{K_\Gamma\|\nabla h\|_\infty}{\lambda}|\mathsf L_xv|
\le
\frac{K_\Gamma C_{\mathsf L}\|\nabla h\|_\infty}{\lambda}|v|.
\]
If \(i\in I(x)\), then Lemma \ref{lem:projection} gives \(\mathsf L_xR_i=0\), and hence $\ell_x(R_i)=\Lambda_x(\mathsf L_xR_i)=0.$ This completes the proof.
\end{proof}
\begin{remark}[Regularity at the boundary]\label{rem:no-classical-gradient} At a boundary point $x$, the proposition gives a bounded linear functional $\ell_x$ extending the feasible one-sided derivative.  Since $g$ is Lipschitz and $w\mapsto\ell_x(w)$ is linear, the ray derivatives imply the cone-wise first-order expansion
\[
 g(x+\epsilon w)=g(x)+\epsilon\ell_x(w)+o_w(\epsilon)
 \qquad
 \text{as }\epsilon\to0,\ x+\epsilon w\in E.
\]
This expansion is local at the fixed boundary point.  It does not assert that $x\mapsto\ell_x$ is continuous towards the boundary, nor that the interior gradient $\nabla g(y)$ has a limit as $y\to x$ from $E^\circ$.  The smoothing argument uses only the algebraic value $\ell_x(R_i)=0$ on active reflection directions, after the feasible-direction limit has been averaged against the one-sided mollifier.
\end{remark}

\begin{proof}[Proof of \cref{prop:resolvent-regularity-projected}]
Boundedness and Lipschitz continuity are \cref{lem:resolvent-Lipschitz}. The statement that $g$ is a classical solution of the resolvent equation in $E^\circ$ is \cref{lem:interior-PDE}. The projection formula and the identity $\mathsf L_xR_i=0$ for $i\in I(x)$ are \cref{lem:projection}. Finally, \cref{prop:resolvent-directional} gives the feasible one-sided derivatives, the factorization through $\mathsf L_x$, and the linear extension $\ell_x$ satisfying $\ell_x(R_i)=0$ on active faces.
\end{proof}

\section{One-sided smoothing and the measure--Neumann approximation}\label{sec:smoothing}

In this section we prove the resolvent insertion theorem,
\cref{thm:resolvent-insertion}.  We first prove the
identity for $h\in C_c^\infty(\mathbb R^d)$, regarded as a function on $E$, by
a one-sided smoothing argument.  At the end of the proof, a density argument
extends the identity to all $h\in C_0(E)$.  Thus, until this final density step throughout this whole section,
we fix $\lambda>0$, $h\in C_c^\infty(\mathbb R^d)$, and write
$g=R_\lambda h$ as in \eqref{eq:resolvent}.

\begin{theorem}[Resolvent insertion theorem]\label{thm:resolvent-insertion}
For every finite signed BAR tuple $(\bar\pi,\bar\nu_1,\ldots,\bar\nu_d)$, every $h\in C_0(E)$, and every $\lambda>0$,
\begin{equation}\label{eq:resolvent-invariance}
\tag{RI}
 \int_E(\lambda R_\lambda h-h)\,d\bar\pi=0.
\end{equation}
\end{theorem}

First we prove the insertion for smooth compactly supported $h$, then extend it to $C_0(E)$ by uniform approximation. A convolution supported strictly inside the orthant produces bounded $C^2$ functions on a neighborhood of the closed state space. Integration by parts in the convolution variable proves vanishing of every oblique boundary derivative, with a bound uniform in the smoothing scale. Since the BAR is imposed on $C_b^2(E)$, no spatial cutoff is needed in the resolvent insertion.

Now we one-sided smoothing the function $g$. We seek a mollifier $\rho\in C_c^\infty((1,2)^d)$ that satisfies $\rho\ge0$ and $\int_{\R^d}\rho(w)\,dw=1.$ For $\varepsilon>0$, define
\begin{equation}\label{eq:one-sided-smoothing}
 g_\varepsilon(x)=\int_{\R^d}\rho(w)g(x+\varepsilon w)\,dw,\qquad h_\varepsilon(x)=\int_{\mathbb R^d}\rho(w)h(x+\varepsilon w)\,dw ,
 \qquad x\in E.
\end{equation}
Because the support of $\rho$ lies strictly inside the positive orthant, $g_\varepsilon$ is defined and smooth on an open neighborhood of $E$.

\begin{lemma}[One-sided smoothing]\label{lem:smoothing-basic}
$g_\varepsilon\in C_b^2(E)$ for each fixed $\varepsilon>0$, $g_\varepsilon$ is smooth on an open neighborhood of $E$, $g_\varepsilon\to g$ uniformly, $g_\varepsilon$ is uniformly bounded and Lipschitz, and
\begin{equation}\label{eq:smoothed-PDE}
 (\lambda-L)g_\varepsilon=h_\varepsilon
 \qquad\text{on }E.
\end{equation}
Moreover,
\begin{align}
 \norm{g_\varepsilon-g}_\infty&\le C_\rho\varepsilon\Lip(g),
 \label{eq:g-uniform-convergence}\\
 \norm{h_\varepsilon-h}_\infty&\longrightarrow0,
 \label{eq:h-uniform-convergence}\\
 \norm{g_\varepsilon}_\infty&\le\norm{g}_\infty,
 \label{eq:g-epsilon-bound}\\
 \norm{\nabla g_\varepsilon}_\infty&\le C_\rho\Lip(g),
 \label{eq:gradient-bound}\\
 \norm{D^2g_\varepsilon}_\infty&\le C_\rho\varepsilon^{-1}\Lip(g).
 \label{eq:hessian-bound}
\end{align}
\end{lemma}

\begin{proof} We first justify the smoothness of \(g_\varepsilon\). Let $\delta_\rho:=\operatorname{dist}(\operatorname{supp}\rho,\partial \mathbb R_+^d)>0 .$
For each fixed \(\varepsilon>0\), define $U_\varepsilon
        :=
        \{x\in\mathbb R^d: x+\varepsilon w\in E^\circ
          \text{ for every } w\in\operatorname{supp}\rho\}.$ Then \(U_\varepsilon\) is an open neighborhood of \(E\), because
\(\operatorname{supp}\rho\Subset(0,\infty)^d\). Thus \(g_\varepsilon\) is well defined on
\(U_\varepsilon\). Although \(g\) is only known to be
globally Lipschitz on \(E\), the derivatives of \(g_\varepsilon\) may be computed by
integration by parts in the convolution variable. For every multiindex \(\alpha\),
\[
        \partial_x^\alpha g_\varepsilon(x)
        =
        (-1)^{|\alpha|}\varepsilon^{-|\alpha|}
        \int_{\mathbb R^d}
        \partial_w^\alpha\rho(w)\,g(x+\varepsilon w)\,dw,
        \qquad x\in U_\varepsilon .
\tag{4.10}
\]
This identity is first obtained in the sense of distributions on \(U_\varepsilon\). Since
the right hand side is continuous in \(x\), it is the classical derivative. Iterating the
same argument gives derivatives of all orders; hence \(g_\varepsilon\in C^\infty(U_\varepsilon)\).
In particular, \(g_\varepsilon\in C^2(E)\) in the closed domain sense.

For fixed \(x\in E\), the compact set \(x+\varepsilon\operatorname{supp}\rho\) lies in
\(E^\circ\). Same as the proof of Lemma \ref{lem:interior-PDE}, \(g\) is a classical solution of $(\lambda-L)g=h$ on this compact subset of the interior. Since \(L\) has constant coefficients, differentiating
under the integral on this interior compact set gives $(\lambda-L)g_\varepsilon(x)
        =
        \int_{\mathbb R^d}\rho(w)(\lambda-L)g(x+\varepsilon w)\,dw  =
        \int_{\mathbb R^d}\rho(w)h(x+\varepsilon w)\,dw
        =
        h_\varepsilon(x),$ which proves \eqref{eq:smoothed-PDE}.

Next, the global Lipschitz continuity of \(g\) gives
\[
 \abs{g_\varepsilon(x)-g(x)}
 \le\varepsilon\Lip(g)\int\abs{w}\rho(w)\,dw,
\]
which is \eqref{eq:g-uniform-convergence}. Since \(h\in C_c^\infty(\mathbb R^d)\), it is uniformly continuous, and therefore $\|h_\varepsilon-h\|_\infty\to 0$ which proves gives \eqref{eq:h-uniform-convergence}.
The bound \eqref{eq:g-epsilon-bound} follows from $\rho\ge0$ and $\int\rho=1$:
\[
        |g_\varepsilon(x)|
        \le
        \int_{\mathbb R^d}\rho(w)|g(x+\varepsilon w)|\,dw
        \le
        \|g\|_\infty .
\]

For the gradient, integration by parts in $w$ and $\int_{\mathbb R^d}\partial_{w_j}\rho(w)\,dw=0$  gives
\begin{equation}\label{eq:partial-mollifier}
 \partial_{x_j}g_\varepsilon(x)
 =-\frac1\varepsilon\int \partial_{w_j}\rho(w)
 \bigl(g(x+\varepsilon w)-g(x)\bigr)\,dw.
\end{equation}
Hence $|\partial_{x_j}g_\varepsilon(x)| \le \operatorname{Lip}(g) \int_{\mathbb R^d} |\partial_{w_j}\rho(w)|\,|w|\,dw .$ The Lipschitz bound on the difference proves \eqref{eq:gradient-bound}. Differentiating once more in the same distributional-convolution formula and subtracting the constant $g(x)$ gives
\[
 \partial_{x_jx_k}^2g_\varepsilon(x)
 =\frac1{\varepsilon^2}\int \partial_{w_jw_k}^2\rho(w)
 \bigl(g(x+\varepsilon w)-g(x)\bigr)\,dw,
\]
and therefore $|\partial_{x_jx_k}^2g_\varepsilon(x)|\le
\varepsilon^{-1}\operatorname{Lip}(g)\int_{\mathbb R^d} |\partial_{w_jw_k}^2\rho(w)|\,|w|\,dw .$ Taking the maximum over \(j,k\) proves \eqref{eq:hessian-bound}. Thus $g_\varepsilon$, its first derivatives, and its second derivatives are bounded on $E$, so $g_\varepsilon\in C_b^2(E)$.
\end{proof}

The next proposition is where the projected derivative information for $g$ is used at the boundary. It proves two facts on each face $F_i$: first, $D_i g_\varepsilon(x)\to0$ for every $x\in F_i$; second, the uniform bound in \eqref{eq:boundary-uniform-bound} holds. Together these imply the boundary measure convergence in \eqref{eq:boundary-measure-convergence} by dominated convergence. The pointwise limit is obtained from \cref{prop:resolvent-directional}: if $x\in F_i$ and $w\in\supp\rho\subset(0,\infty)^d$, then $w\in G_x$ and $(g(x+\varepsilon w)-g(x))/\varepsilon\to\ell_x(w)$. After integration by parts in the smoothing variable, the limit of $D_i g_\varepsilon(x)$ becomes $\ell_x(R_i)$, which is zero because $\mathsf L_xR_i=0$ on active faces.

\begin{proposition}[Vanishing oblique derivative after one sided smoothing]\label{prop:vanishing-derivative}
For every \(i\in J\) and every \(x\in F_i\), $D_i g_\varepsilon(x)\longrightarrow 0$ as  $\varepsilon\downarrow0.$ At the same time, the convergence is pointwise in \(x\). There is a constant
\(C_{\rho,R}<\infty\), independent of \(x\) and \(\varepsilon\), such that
\begin{equation}\label{eq:boundary-uniform-bound}
|D_i g_\varepsilon(x)|
\le
C_{\rho,R}\operatorname{Lip}(g),
\qquad
x\in F_i,\quad 0<\varepsilon<1 .
\end{equation}
Consequently, for every finite signed measure \(\eta_i\) on \(F_i\) and every
bounded Borel function \(a:F_i\to\mathbb R\) fixed independently of
\(\varepsilon\), we have 
\begin{equation}\label{eq:boundary-measure-convergence}
    \int_{F_i} a(x)D_i g_\varepsilon(x)\,d\eta_i(x)
\longrightarrow 0 .
\end{equation}
\end{proposition}

\begin{proof}
Since \(\rho\in C_c^\infty((1,2)^d)\), integration by parts in the
\(w\)-variable gives, for \(x\in E\),
\[
\nabla g_\varepsilon(x)
=
-\frac1{\varepsilon}
\int_{\mathbb R^d}
\nabla\rho(w)\,g(x+\varepsilon w)\,dw .
\]
Also, $\int_{\mathbb R^d} R_i\cdot\nabla\rho(w)\,dw=0.$ Therefore, for \(x\in F_i\),
\begin{equation}
D_i g_\varepsilon(x)
=
-\int_{\mathbb R^d}
(R_i\cdot\nabla\rho)(w)
\frac{g(x+\varepsilon w)-g(x)}{\varepsilon}\,dw .
\end{equation}

Fix \(x\in F_i\). Then \(i\in I(x)\). Since
\(\operatorname{supp}\rho\subset (1,2)^d\), every
\(w\in\operatorname{supp}\rho\) belongs to \(G_x\). Proposition \ref{prop:resolvent-directional} gives,
for each such fixed \(w\),
\(
\frac{g(x+\varepsilon w)-g(x)}{\varepsilon}
\longrightarrow
\ell_x(w).\) Moreover,
\[
\left|
(R_i\cdot\nabla\rho)(w)
\frac{g(x+\varepsilon w)-g(x)}{\varepsilon}
\right|
\le
|(R_i\cdot\nabla\rho)(w)|\,\operatorname{Lip}(g)|w|.
\]
The right hand side is integrable over \(\mathbb R^d\), because
\(\rho\) is smooth and compactly supported. Dominated convergence in the
mollifier variable \(w\) gives
\begin{equation}
\lim_{\varepsilon\downarrow0}D_i g_\varepsilon(x)
=
-\int_{\mathbb R^d}
(R_i\cdot\nabla\rho)(w)\ell_x(w)\,dw .
\end{equation}
Since \(\ell_x\) is linear and \(\rho\) has compact support, another
integration by parts gives
\begin{equation}
-\int_{\mathbb R^d}
(R_i\cdot\nabla\rho)(w)\ell_x(w)\,dw
=
\int_{\mathbb R^d}\rho(w)\ell_x(R_i)\,dw
=
\ell_x(R_i).
\end{equation}
There is no boundary term because \(\rho\in C_c^\infty((1,2)^d)\).
Since \(i\in I(x)\), Proposition \ref{prop:resolvent-directional} gives \(\ell_x(R_i)=0\). This proves
the pointwise convergence.

The same representation and the Lipschitz bound give
\[
|D_i g_\varepsilon(x)|
\le
\operatorname{Lip}(g)
\int_{\mathbb R^d}
|(R_i\cdot\nabla\rho)(w)|\,|w|\,dw .
\]
Thus the uniform estimate holds with $C_{\rho,R}
:=
\max_{1\le k\le d}
\int_{\mathbb R^d}
|(R_k\cdot\nabla\rho)(w)|\,|w|\,dw
<\infty .$

Finally, let \(\eta_i\in M(F_i)\), and let
\(a:F_i\to\mathbb R\) be bounded Borel and fixed independently of
\(\varepsilon\). Since \(D_i g_\varepsilon\) is continuous on \(F_i\), the
product \(aD_i g_\varepsilon\) is Borel. The pointwise convergence just proved
and the bound
\[
|a(x)D_i g_\varepsilon(x)|
\le
\|a\|_\infty C_{\rho,R}\operatorname{Lip}(g)
\]
allow dominated convergence with respect to \(|\eta_i|\). Hence
\(\int_{F_i}a(x)D_i g_\varepsilon(x)\,d\eta_i(x)
\longrightarrow0 .\)
This completes the proof.
\end{proof}

\begin{proposition}[Measure--Neumann resolvent approximation]\label{prop:measure-Neumann}
Let $m$ be a finite signed measure on $E$ and let $\eta_i$ be finite signed measures on $F_i$. Then, as $\varepsilon\downarrow0$,
\begin{align}
 \int_E(\lambda-L)g_\varepsilon\,dm
 &\longrightarrow\int_Eh\,dm,
 \label{eq:MN-interior}\\
 \int_Eg_\varepsilon\,dm
 &\longrightarrow\int_Eg\,dm,
 \label{eq:MN-function}\\
 \int_{F_i}D_i g_\varepsilon\,d\eta_i
 &\longrightarrow0,
 \qquad i=1,\ldots,d.
 \label{eq:MN-boundary}
\end{align}
\end{proposition}

\begin{proof}
The first assertion follows from $(\lambda-L)g_\varepsilon=h_\varepsilon$ and the uniform convergence $h_\varepsilon\to h$. The second follows from the uniform convergence $g_\varepsilon\to g$. Since $m$ is finite signed, uniform convergence is sufficient in both cases.
For the boundary terms, take the bounded Borel multiplier $a\equiv1$ in \eqref{eq:boundary-measure-convergence}. This gives \eqref{eq:MN-boundary} for each finite signed boundary measure $\eta_i$.
\end{proof}

\begin{proof}[Proof of \cref{thm:resolvent-insertion}]
First assume $h\in C_c^\infty(\mathbb R^d)$, regarded as a function on $E$, and let $g=R_\lambda h$. By \cref{lem:smoothing-basic}, $g_\varepsilon\in C_b^2(E)$, so it is an admissible BAR test. Applying the BAR to $g_\varepsilon$ and rearranging gives
\begin{equation}\label{eq:BAR-resolvent-rearranged}
 \int_E(\lambda-L)g_\varepsilon\,d\bar\pi
 =\lambda\int_Eg_\varepsilon\,d\bar\pi
 +\sum_{i=1}^d\int_{F_i}D_i g_\varepsilon\,d\bar\nu_i.
\end{equation}
By \cref{prop:measure-Neumann}, letting $\varepsilon\downarrow0$ yields
\[
 \int_E h\,d\bar\pi
 =\lambda\int_E R_\lambda h\,d\bar\pi.
\]
Equivalently, $\int_E(\lambda R_\lambda h-h)\,d\bar\pi=0$ holds for every smooth compactly supported $h$.

Now let $h\in C_0(E)$. The restrictions to $E$ of functions in $C_c^\infty(\mathbb R^d)$ are uniformly dense in $C_0(E)$: extend a function from the closed set $E$ to $C_0(\mathbb R^d)$, cut it off, and mollify on $\mathbb R^d$. Choose $h_n\in C_c^\infty(\mathbb R^d)$ with $\|h_n-h\|_\infty\to0$ on $E$. Since $(P_t)$ is a contraction on bounded functions,
\[
 \|R_\lambda(h_n-h)\|_\infty
 \le \lambda^{-1}\|h_n-h\|_\infty.
\]
The finiteness of $\bar\pi$ therefore permits passage to the limit in the smooth identity, proving \eqref{eq:resolvent-invariance} for $h\in C_0(E)$.
\end{proof}

\section{Proof of \cref{prop:RI-criterion}: From the resolvent identity to signed BAR uniqueness}\label{sec:completion}

We now complete the proof of resolvent identity criterion (\cref{prop:RI-criterion}) and then complete the proof of the main Theorem (\cref{thm:main}). We first use the uniqueness of Laplace transforms to show that the signed measure that satisfies \eqref{eq:RI} is invariant for the semigroup \((P_t)\).

\begin{lemma}[Uniqueness of Laplace transforms{\normalfont\cite[Chapter~II]{Doetsch1974}}]\label{lem:laplace-uniqueness}
Let \(a:[0,\infty)\to\mathbb R\) be locally integrable and of at most exponential
growth. Suppose that its Laplace transform
\(
        \widehat a(\lambda)
        =
        \int_0^\infty e^{-\lambda t}a(t)\,dt
\)
vanishes for every \(\lambda\) in some interval \((\lambda_\ast,\infty)\). Then
\(a(t)=0\text{ for Lebesgue almost every }t\ge0 .\)
\end{lemma}

\begin{proposition}[Resolvent identity implies semigroup invariance]\label{prop:semigroup-invariance}
Assume that a finite signed measure \(\bar\pi\) satisfies \eqref{eq:RI} for every
\(h\in C_0(E)\) and every \(\lambda>0\). Then \(\bar\pi\) is invariant for \((P_t)\):
\begin{equation}\label{eq:signed-invariance}
        \int_E P_t\varphi\,d\bar\pi
        =
        \int_E \varphi\,d\bar\pi,
        \qquad
        t\ge0,\quad \varphi\in C_0(E).
\end{equation}
\end{proposition}

\begin{proof}
From \eqref{eq:RI}, for every \(h\in C_0(E)\) and every \(\lambda>0\),
\begin{equation}\label{eq:RI2}
    \int_E R_\lambda h\,d\bar\pi
        =
        \lambda^{-1}\int_E h\,d\bar\pi .
\end{equation}
Fix \(h\in C_0(E)\) and define $F_h(t)
        =
        \int_E P_t h\,d\bar\pi \text{ }(t>0).$ We first record the elementary regularity of \(F_h\). By the Feller statement \textup{(ii)} in \cref{thm:regularity-SRBM}, \(P_tC_0(E)\subset C_0(E)\), and $\|P_t h-h\|_\infty\xrightarrow{t\downarrow0} 0.$ The semigroup property and the contraction property imply norm continuity of
\(t\mapsto P_t h\) on all of \([0,\infty)\). Indeed, $\|P_s h-P_t h\|_\infty
        =
        \|P_{\min\{s,t\}}(P_{|s-t|}h-h)\|_\infty
        \le
        \|P_{|s-t|}h-h\|_\infty
        \longrightarrow0$ as \(s\to t\).

Since \(\bar\pi\) is finite signed, it follows that \(F_h\) is continuous:
\[
        |F_h(s)-F_h(t)|
        \le
        |\bar\pi|(E)\,\|P_s h-P_t h\|_\infty .
\]
Moreover, we have $|F_h(t)|
        \le
        \|h\|_\infty |\bar\pi|(E),$ for \(|P_t h|\le\|h\|_\infty\). We now pass from the resolvent identity to a Laplace transform identity. Since
\(
        |e^{-\lambda t}P_t h(x)|
        \le
        e^{-\lambda t}\|h\|_\infty
\)
and
\(
        \int_E\int_0^\infty e^{-\lambda t}\|h\|_\infty\,dt\,d|\bar\pi|(x)
        =
        \lambda^{-1}\|h\|_\infty|\bar\pi|(E)
        <
        \infty,
\)
Fubini's theorem gives
\[
\begin{aligned}
        \int_E R_\lambda h\,d\bar\pi
        &=
        \int_E\int_0^\infty e^{-\lambda t}P_t h(x)\,dt\,d\bar\pi(x)  \\
        &=
        \int_0^\infty e^{-\lambda t}
        \left(\int_E P_t h\,d\bar\pi\right)dt                      =
        \int_0^\infty e^{-\lambda t}F_h(t)\,dt .
\end{aligned}
\]
At the same time, since $\lambda^{-1}\int_E h\,d\bar\pi
        =
        \int_0^\infty e^{-\lambda t}F_h(0)\,dt ,$ therefore \eqref{eq:RI2} is equivalent to
\[
        \int_0^\infty e^{-\lambda t}
        \bigl(F_h(t)-F_h(0)\bigr)\,dt
        =
        0,
        \qquad \lambda>0 .
\tag{5.3}
\]
Put $a_h(t)=F_h(t)-F_h(0).$ Then \(a_h\) is continuous and bounded. In particular, \(a_h\) is locally integrable and
of at most exponential growth. Indeed,
\[
        |a_h(t)|
        \le
        |F_h(t)|+|F_h(0)|
        \le
        2\|h\|_\infty|\bar\pi|(E),
        \qquad t\ge0 .
\]
Equation \((5.3)\) says precisely that the Laplace transform of \(a_h\) vanishes for
every \(\lambda>0\). By \cref{lem:laplace-uniqueness},
\(a_h(t)=0 \text{ for Lebesgue almost every }t\ge0 .\) Since \(a_h\) is continuous, this almost everywhere equality upgrades to equality for
every \(t\ge0\). Hence $F_h(t)=F_h(0)$ for $t\ge 0$, \emph{i.e.} $\int_E P_t h\,d\bar\pi
        =
        \int_E h\,d\bar\pi,$ holds for all  $h\in C_0(E)$. 
        
        For a finite
signed measure \(\alpha\), we write $(\alpha P_t)(B)
        =
        \int_E P_t(x,B)\,\alpha(dx)$ for  $B\in\mathcal B(E).$ Then \(\bar\pi P_t\) is a finite signed Radon measure, and for every
\(\varphi\in C_0(E)\),
\[
        \int_E \varphi\,d(\bar\pi P_t)
        =
        \int_E P_t\varphi\,d\bar\pi
        =
        \int_E \varphi\,d\bar\pi .
\]
Since \(C_0(E)\) separates finite Radon measures on the locally compact space \(E\), this
implies $\bar\pi P_t=\bar\pi.$
\end{proof}

\cref{prop:semigroup-invariance} indicates that the signed-BAR interior solution $\bar\pi$ is invariant as a signed measure for the SRBM semigroup, i.e. $\bar\pi P_t=\bar\pi$.
Then we follows the Dai-Dieker \cite{DaiDieker2011}, using Jordan decomposition to show that $\bar\pi$ uniquely characterizes the
stationary probability distribution $\pi_0$ of the diffusion process in the sense that $\bar\pi=\bar\pi(E)\pi_0$. 

\begin{proposition}[Identification of the interior measure]\label{prop:interior-identification}
If $\bar\pi$ is a finite signed invariant measure for the SRBM semigroup, then $\bar\pi=c\pi_0$ where $c=\bar\pi(E)$.
\end{proposition}

\begin{proof}
For a Markov kernel $P$ and a finite signed measure $\bar\pi$, positivity gives the measure inequality
\begin{equation}\label{eq:variation-kernel}
 |\bar\pi P|\le |\bar\pi|P.
\end{equation}
Indeed, for every Borel set $B$,
$|\int P(x,B)\,d\bar\pi(x)|\le\int P(x,B)\,d|\bar\pi|(x)$, and the same domination holds for finite measurable partitions, hence for total variation. If $\bar\pi P_t=\bar\pi$, then \eqref{eq:variation-kernel} gives $|\bar\pi|\le|\bar\pi|P_t$. Both positive measures have total mass $|\bar\pi|(E)$, because $P_t(x,E)=1$. Thus the domination is actually equality: if finite positive measures $\alpha\le\beta$ have $\alpha(E)=\beta(E)$, then $\beta-\alpha$ is a positive measure with total mass zero, hence vanishes. Applying this with $\alpha=|\bar\pi|$ and $\beta=|\bar\pi|P_t$ gives
\(
 |\bar\pi|P_t=|\bar\pi|.
\)

Consequently the Jordan components $\bar\pi^+=\tfrac12(|\bar\pi|+\bar\pi)$ and $\bar\pi^-=\tfrac12(|\bar\pi|-\bar\pi)$ are invariant positive finite measures. Each nonzero component, after normalization by its total mass, is an invariant probability and therefore equals $\pi_0$, which is unique. Thus
\(
 \bar\pi=(\bar\pi^+(E)-\bar\pi^-(E))\pi_0=\bar\pi(E)\pi_0.
\)
\end{proof}

The preceding \cref{prop:semigroup-invariance,prop:interior-identification} identifies the uniqueness of the interior measure, but an exact same assessment for boundary measure is not yet established.
Therefore, by subtracting the appropriate scalar multiple of the stationary BAR vector, the
remaining signed tuple has zero interior measure. Thus the only possible
obstruction to the signed uniqueness of BAR solution is a purely boundary one: a collection of
finite signed measures \((\eta_i)_{i=1}^d\), with \(\eta_i\) supported on \(F_i\),
whose boundary pairing vanishes, i.e. $\sum_{i=1}^d \int_{F_i} D_i f\,d\eta_i=0$ against every test function $f\in C_b^2(E)$. If such boundary measures exist, adding them to an existing solution doesn't change the validity of the solution.
The next proposition shows that no such nontrivial boundary annihilator exists, ruling out any other signed-BAR solutions.

The proof is local on the boundary stratification and uses the nonsingular
\(M\)-matrix assumption only through the invertibility of the principal
reflection blocks \(R_{AA}\). On a stratum \(S_A\), the active oblique derivatives
are determined by the active normal jet:
   $(D_i f|_{S_A})_{i\in A}=R_{AA}^T a$, where
   $a=(\partial_{x_j}f|_{S_A})_{j\in A}$.
Since \(R_{AA}\) is invertible, we can prescribe these oblique derivatives
independently. In particular, choosing
   $a=R_{AA}^{-T}e_k\,\psi$
gives
   $D_i f|_{S_A}=\delta_{ik}\psi$ for $i\in A$, and
for any test function \(\psi\in C_c^\infty(S_A)\). This isolates the \(k\)-th boundary
measure on \(S_A\). An induction over the codimension \(|A|\) removes all
lower-stratum contributions and forces \(\eta_k|_{S_A}=0\). Since both \(A\)
and \(k\in A\) are arbitrary, all boundary measures vanish.

\begin{proposition}[Pure boundary injectivity]\label{prop:pure-boundary}
Let $\eta_i\in\cM(F_i)$ be finite signed measures. If
\begin{equation}\label{eq:pure-boundary}
 \sum_{i=1}^d\int_{F_i}D_i f\,d\eta_i=0,
 \qquad f\in C_b^2(E),
\end{equation}
then $\eta_i=0,$ for all $i=1\ldots,d.$
\end{proposition}

\begin{proof}
For $A\subset J$ and $i\in A$, let $ \eta_i^A=\eta_i|_{S_A}.$ Then $\eta_i=\sum_{A\ni i}\eta_i^A$ as a finite sum of mutually singular signed measures. We prove by induction on $n=|A|$ that
\begin{equation}\label{eq:induction-claim}
 \eta_i^A=0
 \quad\text{for every }A\subset J\text{ with }|A|=n
 \text{ and every }i\in A.
\end{equation}

Assume the claim has been proved for all strata of cardinality less than $n$, and fix $A$ with $|A|=n$. Let $k\in A$ and let $\psi\in C_c^\infty(S_A)$. Write points as $x=(x_A,y)$, where $y=x_{A^c}\in(0,\infty)^{A^c}$. We identify $\psi$ with its extension by zero to $\R^{A^c}$; this extension is smooth because $\supp\psi$ is compact in the open orthant. Choose a tangential cutoff $\vartheta\in C_c^\infty((0,\infty)^{A^c})$ that equals one on a neighborhood of $\supp\psi$. When $A=J$, interpret the tangential space as a point and put $\vartheta=1$. Choose a normal cutoff $\zeta\in C_c^\infty(\R^A)$ that equals one near the origin. Define the $A$-vector
\begin{equation}\label{eq:a-vector}
 a(y)=R_{AA}^{-T}e_k\,\psi(y)
\end{equation}
and the test function
\begin{equation}\label{eq:boundary-test}
 f(x_A,y)=\zeta(x_A)\vartheta(y)
 \sum_{j\in A}x_j a_j(y).
\end{equation}
After shrinking the support of $\vartheta$ if necessary, $f$ is supported away from every face $F_j$ with $j\notin A$. At a point of $S_A$, the tangential derivatives of $f$ vanish and
\begin{equation}\label{eq:active-jet}
 \partial_{x_j}f(0,y)=a_j(y),
 \qquad j\in A.
\end{equation}
Therefore, for $i\in A$,
\begin{equation}\label{eq:prescribed-oblique-jet}
 D_i f|_{S_A}
 =\sum_{j\in A}R_{ji}a_j
 =(R_{AA}^Ta)_i
 =\delta_{ik}\psi.
\end{equation}

The support condition implies that the only boundary strata meeting $\supp f$ are $S_C$ with $\varnothing\ne C\subset A$. The contributions from $|C|<n$ vanish by the induction hypothesis. Hence \eqref{eq:pure-boundary} and \eqref{eq:prescribed-oblique-jet} give
\[
 0=\sum_{i\in A}\int_{S_A}D_i f\,d\eta_i^A
 =\int_{S_A}\psi\,d\eta_k^A.
\]
Since $\psi$ is arbitrary, $\eta_k^A=0$. Since $k\in A$ was arbitrary, the induction step is complete. The base case $n=1$ is the same argument with no lower strata. Thus all $\eta_i^A$ vanish and hence all $\eta_i$ vanish.
\end{proof}

Now, the uniqueness of Signed BAR solution in the Harrison-Reiman Class is the natural conclusion of \cref{prop:semigroup-invariance}-\cref{prop:pure-boundary}.

\begin{proof}[Proof of \cref{prop:RI-criterion}]
Assume \eqref{eq:RI} for the finite signed BAR tuple $(\bar\pi,\bar\nu_1,\ldots,\bar\nu_d)$. By \cref{prop:semigroup-invariance}, we know the signed measure $\bar\pi$ is invariant under the semigroup, \emph{i.e.} $\bar\pi P_t=\bar\pi$ holds for $t\ge 0$. Then by \cref{prop:interior-identification}, we have the uniquness of the interior measure $\bar\pi=c\pi_0$ for $c=\bar\pi(E)$.

Define $\eta_i=\bar\nu_i-c\nu_i^0.$ Subtract $c$ times the stationary BAR \eqref{eq:intro-stationary-BAR} from the signed BAR \eqref{eq:signed-BAR}. The interior measures cancel, and we obtain $\sum_{i=1}^d\int_{F_i}D_i f\,d\eta_i=0,$ for all test function $f\in C_b^2(E).$ \Cref{prop:pure-boundary} gives $\eta_i=0$ for every $i$, hence $\bar\nu_i=c\nu_i^0$. This proves \cref{prop:RI-criterion}.
\end{proof}

Finally, we finish the proof of the signed-BAR problem.
\begin{proof}[Proof of \cref{thm:main}]
\Cref{thm:resolvent-insertion} proves \eqref{eq:RI} for every finite signed BAR tuple. \Cref{prop:RI-criterion} converts \eqref{eq:RI} into full signed BAR uniqueness. Therefore \cref{thm:main} follows.
\end{proof}

\section{Failure in the completely \texorpdfstring{$S$}{S} class}\label{sec:sharpness}

%The preceding proof uses two matrix facts that are automatic for nonsingular $M$-matrices: the derivative projection is defined by $R_{AA}^{-1}$ on every active set, and the pure boundary identity is killed by prescribing oblique jets through $R_{AA}^{-T}$. This section shows that these facts are not technical conveniences. In the larger completely-$\mathcal{S}$ class, a singular proper principal block gives an explicit signed boundary gauge. After a zero-potential correction, the gauge becomes a finite signed BAR tuple whose interior coordinate has total mass zero but is not zero.

In this section, we demonstrate a family of counterexamples in the completely $\mathcal{S}$ class due to the singular proper active block of the relection matrix. The negative construction is summarized in \cref{fig:counterexample-architecture}.  It has two parts: a boundary gauge identity on the singular stratum, followed by a zero-potential correction that extends the resulting centered source into the interior and adds the matching boundary occupation potential. Recall that a square matrix $A$ is an $S$ matrix if there exists a vector $u>0$ such that $Au>0$, and is completely-$\mathcal{S}$ if every principal submatrix is an $S$ matrix. completely-$\mathcal{S}$ reflection matrices are the natural existence class for orthant SRBMs, but they need not have invertible principal blocks.

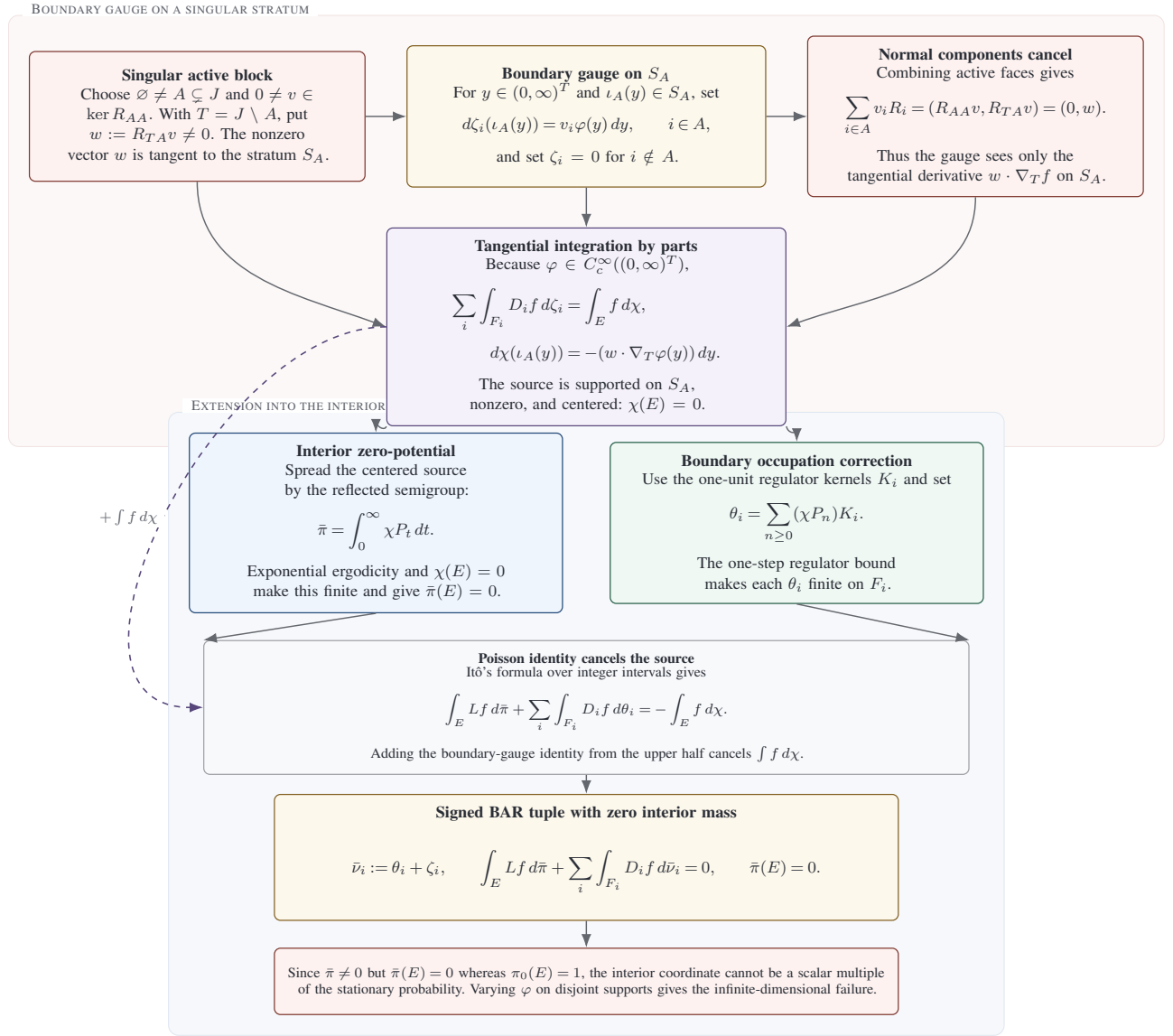
\begin{figure}[t]
\centering
\resizebox{\textwidth}{!}{%
\begin{tikzpicture}[
  x=1cm,y=1cm,
  font=\footnotesize,
  text=proofink,
  >=Latex,
  box/.style={
    draw=proofink!48,
    rounded corners=3pt,
    line width=.55pt,
    fill=white,
    align=center,
    inner xsep=6pt,
    inner ysep=6pt,
    execute at begin node={\hyphenpenalty=10000\exhyphenpenalty=10000\emergencystretch=1.5em}
  },
  algebra/.style={box, draw=proofred!80!black, fill=proofredbg, text width=5.40cm, minimum height=2.25cm},
  gauge/.style={box, draw=proofgold!85!black, fill=proofgoldbg, text width=5.80cm, minimum height=2.45cm},
  source/.style={box, draw=proofpurple!85!black, fill=proofpurplebg, text width=6.50cm, minimum height=2.40cm},
  potential/.style={box, draw=proofblue!78!black, fill=proofbluebg, text width=6.05cm, minimum height=2.45cm},
  correction/.style={box, draw=proofgreen!78!black, fill=proofgreenbg, text width=6.05cm, minimum height=2.45cm},
  combine/.style={box, draw=proofink!45, fill=proofgraybg, text width=12.85cm, minimum height=1.75cm, font=\scriptsize},
  final/.style={box, draw=proofgold!85!black, fill=proofgoldbg, text width=10.40cm, minimum height=1.45cm},
  fail/.style={box, draw=proofred!80!black, fill=proofredbg, text width=10.40cm, minimum height=1.15cm, font=\scriptsize},
  arr/.style={-{Latex[length=2.35mm,width=1.75mm]}, line width=.65pt, draw=proofink!72},
  dasharr/.style={-{Latex[length=2.35mm,width=1.75mm]}, line width=.70pt, dashed, draw=proofpurple!85!black},
  lab/.style={font=\scriptsize, fill=white, inner xsep=3pt, inner ysep=1pt, text=proofmuted},
  glabel/.style={font=\scriptsize\scshape, text=proofmuted, fill=white, inner xsep=4pt, inner ysep=1pt}
]

\node[algebra] (block) at (-6.75,0) {\textbf{Singular active block}\\[-1pt]
Choose $\varnothing\ne A\subsetneq J$ and $0\ne v\in\ker R_{AA}.$
With $T=J\setminus A$, put $w:=R_{TA}v\ne0.$ The nonzero vector $w$ is tangent to the stratum $S_A$.};

\node[gauge] (gauge) at (0,0) {\textbf{Boundary gauge on $S_A$}\\[-1pt]
For $y\in(0,\infty)^T$ and $\iota_A(y)\in S_A$, set
\[
  d\zeta_i(\iota_A(y))=v_i\varphi(y)\,dy,
  \qquad i\in A,
\]
and set $\zeta_i=0$ for $i\notin A$.};

\node[algebra] (cancel) at (6.75,0) {\textbf{Normal components cancel}\\[-1pt]
Combining active faces gives
\[
 \sum_{i\in A}v_iR_i=(R_{AA}v,R_{TA}v)=(0,w).
\]
Thus the gauge sees only the tangential derivative $w\cdot\nabla_T f$ on $S_A$.};

\node[source] (source) at (0,-3.65) {\textbf{Tangential integration by parts}\\[-1pt]
Because $\varphi\in C_c^\infty((0,\infty)^T)$,
\[
\begin{aligned}
 \sum_i\int_{F_i}D_i f\,d\zeta_i&=\int_E f\,d\chi,\\[-1pt]
 d\chi(\iota_A(y))&=-(w\cdot\nabla_T\varphi(y))\,dy .
\end{aligned}
\]
The source is supported on $S_A$, nonzero, and centered: $\chi(E)=0$.};

\node[potential] (pot) at (-3.65,-7.05) {\textbf{Interior zero-potential}\\[-1pt]
Spread the centered source by the reflected semigroup:
\[
  \bar\pi=\int_0^\infty \chi P_t\,dt .
\]
Exponential ergodicity and $\chi(E)=0$ make this finite and give $\bar\pi(E)=0$.};

\node[correction] (theta) at (3.65,-7.05) {\textbf{Boundary occupation correction}\\[-1pt]
Use the one-unit regulator kernels $K_i$ and set
\[
  \theta_i=\sum_{n\ge0}(\chi P_n)K_i .
\]
The one-step regulator bound makes each $\theta_i$ finite on $F_i$.};

\node[combine] (poisson) at (0,-10.25) {\textbf{Poisson identity cancels the source}\\[-1pt]
It\^o's formula over integer intervals gives
\[
  \int_E Lf\,d\bar\pi+\sum_i\int_{F_i}D_i f\,d\theta_i=-\int_E f\,d\chi .
\]
Adding the boundary-gauge identity from the upper half cancels $\int f\,d\chi$.};

\node[final] (tuple) at (0,-12.85) {\textbf{Signed BAR tuple with zero interior mass}\\[-1pt]
\[
  \bar\nu_i:=\theta_i+\zeta_i,
  \qquad
  \int_E Lf\,d\bar\pi+\sum_i\int_{F_i}D_i f\,d\bar\nu_i=0,
  \qquad
  \bar\pi(E)=0.
\]};

\node[fail] (failure) at (0,-15.00) {Since $\bar\pi\ne0$ but $\bar\pi(E)=0$ whereas $\pi_0(E)=1$, the interior coordinate cannot be a scalar multiple of the stationary probability.  Varying $\varphi$ on disjoint supports gives the infinite-dimensional failure.};

\draw[arr] (block.east) -- (gauge.west);
\draw[arr] (gauge.east) -- (cancel.west);
\draw[arr] (cancel.south) to[out=-90,in=20] (source.east);
\draw[arr] (gauge.south) -- (source.north);
\draw[arr] (block.south) to[out=-90,in=160] (source.west);
\draw[arr] (source.south west) to[out=-100,in=90] (pot.north);
\draw[arr] (source.south east) to[out=-80,in=90] (theta.north);
\draw[arr] (pot.south) -- (poisson.north west);
\draw[arr] (theta.south) -- (poisson.north east);
\draw[arr] (poisson.south) -- (tuple.north);
\draw[arr] (tuple.south) -- (failure.north);
\draw[dasharr] (source.west) to[out=180,in=180,looseness=1.05] node[lab,left] {$+\int f\,d\chi$} (poisson.west);

\begin{scope}[on background layer]
  \node[draw=proofred!15, fill=proofred!4, rounded corners=5pt, inner sep=10pt, fit=(block)(gauge)(cancel)(source)] (boundarygrp) {};
  \node[glabel, anchor=south west] at ($(boundarygrp.north west)+(0.25,-0.02)$) {Boundary gauge on a singular stratum};
  \node[draw=proofblue!15, fill=proofblue!4, rounded corners=5pt, inner sep=10pt, fit=(pot)(theta)(poisson)(tuple)(failure)] (interiorgrp) {};
  \node[glabel, anchor=south west] at ($(interiorgrp.north west)+(0.25,-0.02)$) {Extension into the interior by a zero-potential};
\end{scope}

\end{tikzpicture}%
}
\caption{\textbf{Counterexample construction in the completely-$\mathcal{S}$ class.}  The upper group is the boundary algebra of \cref{prop:singular-block-gauge}: a gauge supported on a lower-dimensional stratum has its active normal components killed by $R_{AA}v=0$, leaving a tangential derivative and hence a centered source $\chi$.  The lower group is the zero-potential correction of \cref{prop:zero-potential-correction}: the semigroup potential $\bar\pi=\int_0^\infty \chi P_t\,dt$ and the boundary occupation potentials $\theta_i$ cancel the source and produce a nonzero zero-mass signed BAR tuple.}
\label{fig:counterexample-architecture}
\end{figure}

\subsection{A singular-block boundary gauge}

In this section, we utilize the singular proper active block of the reflection matrix to construct a nonzero null direction on that block and place a signed boundary gauge on the corresponding lower-dimensional boundary stratum. The singularity makes the active normal reflection components cancel, so the gauge leaves only a tangential derivative along the stratum. After integration by parts, this tangential derivative becomes a finite nonzero centered signed source supported on the same stratum. This source will be cancelled by the zero-potential correction in the next subsection.

Let $J=\{1,\ldots,d\}$. For a nonempty proper set $A\subsetneq J$, put $T=J\setminus A$ and write $R_{BC}$ for the submatrix with rows in $B$ and columns in $C$. Define the open stratum
$S_A=\{x\in E:x_i=0\text{ for }i\in A,\ x_j>0\text{ for }j\in T\}.$ We identify $S_A$ with $(0,\infty)^T$ through the embedding $\iota_A:(0,\infty)^T\to E$ that inserts zeros in the coordinates indexed by $A$.

\begin{proposition}[Boundary gauge from a singular principal block]\label{prop:singular-block-gauge}
Assume that $R$ is nonsingular and that $R_{AA}$ is singular for some nonempty proper set $A\subsetneq J$. Choose
\[
 0\ne v\in\ker R_{AA},\qquad w=R_{TA}v\in\mathbb R^T.
\]
Then $w\ne0$. Let $\phi\in C_c^\infty((0,\infty)^T)$ satisfy $w\cdot\nabla_T\phi\not\equiv0$. For $i\in A$, define a finite signed measure $\zeta_i$ on $F_i$ by
\begin{equation}\label{eq:ce-zeta-general}
 \int_{F_i}g(x)\,d\zeta_i(x)
 =v_i\int_{(0,\infty)^T}g(\iota_A(y))\phi(y)\,dy,
\end{equation}
and set $\zeta_i=0$ for $i\notin A$. Define $\chi\in\cM(E)$ by
\begin{equation}\label{eq:ce-source-general}
 \int_E g(x)\,d\chi(x)
 =-\int_{(0,\infty)^T}g(\iota_A(y))
       \bigl(w\cdot\nabla_T\phi(y)\bigr)\,dy.
\end{equation}
Then $\chi$ is finite, nonzero, supported on $S_A$, and satisfies $\chi(E)=0$. Moreover,
\begin{equation}\label{eq:ce-defect-identity}
 \sum_{i=1}^d\int_{F_i}D_i f\,d\zeta_i
 =\int_E f\,d\chi,
 \qquad f\in C_b^2(E).
\end{equation}
\end{proposition}

\begin{proof}
If $w=0$, extend $v$ to $\widetilde v\in\mathbb R^d$ by setting its coordinates in $T$ equal to zero. Then
\[
 R\widetilde v=
 \begin{pmatrix}R_{AA}v\\ R_{TA}v\end{pmatrix}=0,
\]
contradicting the nonsingularity of $R$. Hence $w\ne0$, and a compactly supported smooth $\phi$ with nonzero directional derivative along $w$ exists.

For $f\in C_b^2(E)$, combine the face labels before integrating:
\begin{align*}
\sum_{i=1}^d\int_{F_i}D_i f\,d\zeta_i
 &=\int_{(0,\infty)^T}\phi(y)
       \left(\sum_{i\in A}v_iR_i\right)\cdot
       \nabla f(\iota_A(y))\,dy\\
 &=\int_{(0,\infty)^T}\phi(y)
       \left[(R_{AA}v)\cdot\nabla_A f(\iota_A(y))
             +(R_{TA}v)\cdot\nabla_T f(\iota_A(y))\right]dy\\
 &=\int_{(0,\infty)^T}\phi(y)w\cdot\nabla_T f(\iota_A(y))\,dy\\
 &=-\int_{(0,\infty)^T}
       \bigl(w\cdot\nabla_T\phi(y)\bigr)f(\iota_A(y))\,dy.
\end{align*}
There is no boundary term because $\phi$ is compactly supported in the open stratum. This proves \eqref{eq:ce-defect-identity}. Taking a test function equal to one on a neighborhood of $\supp\phi$ gives $\chi(E)=0$, and the choice of $\phi$ gives $\chi\ne0$.
\end{proof}

The singular block cancels the components normal to the active faces. The remaining vector $w$ is tangent to $S_A$, and tangential integration by parts turns the boundary gauge into the centered source $\chi$.

\subsection{Interior correction by a zero-potential}

In this section, we take that centered source and spread it through the reflected Brownian semigroup by a zero potential, while also adding the matching boundary occupation potentials generated by the regulator. Under exponential ergodicity and a one-step regulator bound, these potentials are finite. Ito’s formula then shows that the zero potential contributes exactly the negative of the source created in the previous section. Adding the original boundary gauge cancels the defect and produces a genuine finite signed BAR tuple. Its interior part has total mass zero but is not the zero measure, so it cannot be a scalar multiple of the stationary distribution, whose mass is one. 

Let $(P_t)_{t\ge0}$ be the transition semigroup of the SRBM, and let $Y=(Y_1,\ldots,Y_d)$ be its regulator. For a finite signed measure $\alpha$, we define the semigroup $(\alpha P_t)(B)=\int_E P_t(x,B)\,\alpha(dx).$ For each face define the  boundary occupation kernel
\begin{equation}\label{eq:ce-Ki}
 K_i(x,B)=\mathbb E_x\int_0^1 \1_B(Z(s))\,dY_i(s),
\end{equation}
which is supported on $F_i$.

\begin{proposition}[Interior correction by a zero-potential]\label{prop:zero-potential-correction}
Assume that the SRBM is a strong Markov process whose transition semigroup is Feller on $C_0(E)$, and that it has stationary probability $\pi_0$. Suppose that there are a locally bounded function $V:E\to[1,\infty)$ and constants $M,\kappa>0$ such that
\begin{equation}\label{eq:ce-ergodicity}
 \|P_t(x,\cdot)-\pi_0\|_{\TV}
 \le M V(x)e^{-\kappa t},
 \qquad x\in E,\ t\ge0.
\end{equation}
Suppose also that
\begin{equation}\label{eq:ce-regulator-bound}
 c_i:=\sup_{x\in E}\mathbb E_xY_i(1)<\infty,
 \qquad i=1,\ldots,d.
\end{equation}
Let $\chi\in\cM(E)$ satisfy $\chi(E)=0$ and $\int_E V\,d|\chi|<\infty,$ and let $\zeta_i\in\cM(F_i)$ satisfy
\begin{equation}\label{eq:ce-boundary-source}
 \sum_{i=1}^d\int_{F_i}D_i f\,d\zeta_i
 =\int_E f\,d\chi,
 \qquad f\in C_b^2(E).
\end{equation}
Define
\begin{align}
 \overline\pi&=\int_0^\infty \chi P_t\,dt,
                                                    \label{eq:ce-pi-potential}\\
 \theta_i&=\sum_{n=0}^\infty(\chi P_n)K_i,
                                                    \label{eq:ce-theta-potential}\\
 \overline\nu_i&=\theta_i+\zeta_i.              \label{eq:ce-nu-corrected}
\end{align}
Then all measures in \eqref{eq:ce-pi-potential}--\eqref{eq:ce-nu-corrected} are well defined in total variation and finite. They form a signed BAR tuple for all $f\in C_b^2(E)$, hence also for all $f\in C_c^2(E)$, and $\overline\pi(E)=0.$ If $\chi\ne0$, then $\overline\pi\ne0$.
\end{proposition}

\begin{proof}
Because $\chi(E)=0$,
\[
 \chi P_t=\int_E\bigl(P_t(x,\cdot)-\pi_0\bigr)\,\chi(dx).
\]
Therefore
\begin{equation}\label{eq:ce-chiPt-bound}
 \|\chi P_t\|_{\TV}
 \le M e^{-\kappa t}\int_E V\,d|\chi|.
\end{equation}
The integral in \eqref{eq:ce-pi-potential} converges in total variation, and $\overline\pi(E)=0$ because $\chi P_t(E)=\chi(E)=0$.

For a finite signed measure $\alpha$ and the positive kernel $K_i$,
\begin{equation}\label{eq:ce-kernel-bound}
 \|\alpha K_i\|_{\TV}
 \le\int_EK_i(x,E)\,|\alpha|(dx)
 \le c_i\|\alpha\|_{\TV}.
\end{equation}
Combining \eqref{eq:ce-chiPt-bound} at integer times with \eqref{eq:ce-kernel-bound} proves absolute convergence of \eqref{eq:ce-theta-potential}. Each $\theta_i$ is supported on $F_i$.

Fix $f\in C_b^2(E)$. It\^o's formula up to an integer time $N$, followed by integration against the Jordan decomposition of $\chi$, gives
\begin{equation}\label{eq:ce-ito-chi}
 \chi P_N(f)-\chi(f)
 =\int_0^N(\chi P_t)(Lf)\,dt
  +\sum_{i=1}^d\int_E\mathbb E_x\int_0^N D_i f(Z(t))\,dY_i(t)\,\chi(dx).
\end{equation}
Splitting the boundary integral into unit intervals and using the strong Markov property yields
\begin{equation}\label{eq:ce-boundary-sum}
 \int_E\mathbb E_x\int_0^N D_i f(Z(t))\,dY_i(t)\,\chi(dx)
 =\sum_{n=0}^{N-1}\bigl((\chi P_n)K_i\bigr)(D_i f).
\end{equation}
The estimates above justify passage to the limit. Since $\chi P_N(f)\to0$, equations \eqref{eq:ce-ito-chi} and \eqref{eq:ce-boundary-sum} give
\[
 \int_E Lf\,d\overline\pi
 +\sum_{i=1}^d\int_{F_i}D_i f\,d\theta_i
 =-\int_E f\,d\chi.
\]
Adding \eqref{eq:ce-boundary-source} proves the BAR for $(\overline\pi,\overline\nu_1,\ldots,\overline\nu_d)$.

It remains to prove that the zero potential is injective on this centered class. Let $\mathcal A$ be the generator of the Feller semigroup on $C_0(E)$. For $h\in D(\mathcal A)$, semigroup differentiation gives
\begin{align*}
 \overline\pi(\mathcal A h)
 =\lim_{T\to\infty}\int_0^T\chi P_t(\mathcal A h)\,dt=\lim_{T\to\infty}\chi(P_T h-h)=-\chi(h),
\end{align*}
where the last limit follows from \eqref{eq:ce-chiPt-bound}. If $\overline\pi=0$, then $\chi(h)=0$ for every $h\in D(\mathcal A)$. To conclude that $\chi=0$ we use the following density theorem.

\begin{theorem}[{\normalfont\cite[Chapter~1, Section~2]{EthierKurtz1986}}]
 The infinitesimal generator $\mathcal A$ of a strongly continuous contraction semigroup on a Banach space has domain $D(\mathcal A)$ dense in that space. In particular, for a Feller semigroup on $C_0(E)$, the domain $D(\mathcal A)$ is dense in $C_0(E)$.    
\end{theorem}

Its hypothesis is exactly the standing assumption of the present proposition: $(P_t)$ is Feller on $C_0(E)$. Hence $D(\mathcal A)$ is dense in $C_0(E)$.  Since \(\mathcal M(E)=C_0(E)^*\), the identity \(\chi(h)=0\) on \(D(\mathcal A)\) implies \(\chi=0\). Therefore \(\chi\ne0\) implies \(\overline\pi\ne0\).
\end{proof}

\begin{theorem}[Failure in the completely-$\mathcal{S}$ class]\label{thm:singular-block-obstruction}
Consider an SRBM in $E=\mathbb R_+^d$ with positive definite covariance matrix, nonsingular completely-$\mathcal{S}$ reflection matrix $R$, and stationary probability $\pi_0$. Assume that the process is strong Markov, is Feller on \(C_0(E)\), and satisfies the quantitative recurrence conditions \eqref{eq:ce-ergodicity} and \eqref{eq:ce-regulator-bound}. If $R_{AA}$ is singular for some nonempty proper set $A\subsetneq J$, then signed BAR uniqueness fails.

More precisely, there exists a finite signed BAR tuple $(\overline\pi,\overline\nu_1,\ldots,\overline\nu_d)$ such that
\begin{equation}\label{eq:ce-nontrivial-zero-mass}
 \overline\pi(E)=0,
 \qquad
 \overline\pi\ne0.
\end{equation}
Consequently $\overline\pi$ is not a scalar multiple of $\pi_0$.
\end{theorem}

\begin{proof}
Apply \cref{prop:singular-block-gauge}. Its source $\chi$ is compactly supported, and the local boundedness of $V$ gives $\int_E V\,d|\chi|<\infty$. \Cref{prop:zero-potential-correction} then produces the required BAR tuple. If $\overline\pi=c\pi_0$, total masses give $c=0$, contradicting $\overline\pi\ne0$.
\end{proof}

\begin{corollary}[Infinite-dimensional failure]\label{cor:infinite-dimensional-failure}
Under the assumptions of \cref{thm:singular-block-obstruction}, the set of interior BAR coordinates of total mass zero contains an infinite-dimensional linear subspace. In particular, the full vector space of finite signed BAR tuples is infinite-dimensional.
\end{corollary}

\begin{proof}
Choose functions $\phi_m\in C_c^\infty((0,\infty)^T)$ with pairwise disjoint supports and $w\cdot\nabla_T\phi_m\not\equiv0$. Let $\chi_m$ and $\overline\pi_m$ be the corresponding sources and zero potentials. If $\sum_{m=1}^N a_m\overline\pi_m=0$, linearity and the injectivity identity in the proof of \cref{prop:zero-potential-correction} imply $\sum_{m=1}^N a_m\chi_m=0$. The sources are nonzero and have pairwise disjoint supports, so every $a_m$ is zero.
\end{proof}

The theorem concerns a singular proper principal block. Singularity of an arbitrary rectangular or nonprincipal submatrix does not yield the cancellation $R_{AA}v=0$ needed in \cref{prop:singular-block-gauge}. Conversely, every principal block of a nonsingular $M$-matrix is nonsingular by \cref{lem:Mmatrix}, so this obstruction is absent from the class covered by \cref{thm:main}.

\subsection{A checkable three-dimensional family}

The general obstruction is useful only if the recurrence assumptions can be verified without solving the stationary distribution. The next criterion is a direct way to do this for a broad positive-reflection subclass. A $Z$ matrix means a matrix with nonpositive off-diagonal entries.

\begin{corollary}[A checkable completely-$\mathcal{S}$ subclass]\label{cor:checkable-completely-S}
Assume in addition that $R_{ii}=1$ and $R_{ij}\ge0$ for all $i,j$. Suppose there is a symmetric positive definite matrix $H$ such that $HR$ is a $Z$ matrix and $H\mu<0$ componentwise. If $R$ is nonsingular, completely $S$, and has a singular proper principal block, then signed BAR uniqueness fails and the conclusion of \cref{cor:infinite-dimensional-failure} holds.
\end{corollary}

\begin{proof}
We verify the standing hypotheses of \cref{thm:singular-block-obstruction}: existence, the strong Markov property, the \(C_0\)-Feller property, the recurrence certificate \eqref{eq:ce-ergodicity}, and the regulator bound \eqref{eq:ce-regulator-bound}.

First, for the existence, the strong Markov property, the \(C_0\)-Feller property, as well as \eqref{eq:ce-regulator-bound}, we have
\begin{proposition}[{\normalfont\cite{TaylorWilliams1993}}]
    For a symmetric positive definite covariance $\Sigma$, a drift $\mu$, and a reflection matrix $R$ with unit diagonal, the orthant SRBM with data $(\Sigma,\mu,R)$ exists and is unique in law if and only if $R$ is completely-$\mathcal S$; when it exists it is a Feller continuous strong Markov process, and $x\mapsto P_t h(x)$ is continuous for every $h\in C_b(E)$. 
\end{proposition}

Since $\Sigma$ is positive definite, $R$ has unit diagonal, and
$R$ is completely-$S$, the orthant SRBM exists, is unique in law, is strong Markov,
and $x\mapsto P_t h(x)$ is continuous for every $h\in C_b(E)$ (Feller continuity).

Now we check the $C_0$ Feller property, i.e. $P_tC_0(E)\subset C_0(E)$ and that
$\|P_t h-h\|_\infty\to0$ for $h\in C_0(E)$. Write $Z^x(t)=x+\mu t+B(t)+RY^x(t),$ where $B(t)=\Sigma^{1/2}W(t)$ and each $Y_i^x$ is nondecreasing. Since $R_{ij}\ge0$ and $Y_j^x\ge0$, the one-dimensional Skorokhod formula gives, for
$0\le s\le t$, we have $Y_i^x(s)
        \le
        |\mu_i|t+\sup_{0\le u\le t}|B_i(u)|.$ (This also verifies \eqref{eq:ce-regulator-bound} since the right side has finite expectation independent of the initial state $x$.) 
        
Hence, with $M_t:=|\mu|_\infty t+\sup_{0\le u\le t}|B(u)|_\infty \text{ and }
        C_R:=1+\max_i\sum_j R_{ij},$ we have the uniform displacement bound
\[
        \sup_{0\le s\le t}|Z^x(s)-x|_\infty \le C_R M_t .
\]

Let $h\in C_c(E)$ and suppose $\operatorname{supp}h\subset\{|y|_\infty\le a\}$.
Then
\[
        |P_t h(x)|
        \le
        \|h\|_\infty
        \mathbb P\{C_RM_t\ge |x|_\infty-a\}
        \longrightarrow 0
        \qquad\text{as } |x|_\infty\to\infty .
\]
Thus $P_t h\in C_0(E)$ for $h\in C_c(E)$. By contraction and approximation of
$C_0(E)$ by compactly supported continuous functions, the same holds for every
$h\in C_0(E)$.

Finally, every $h\in C_0(E)$ is uniformly continuous. If $\omega_h$ is its modulus
of continuity, then
\[
        \sup_{x\in E}|P_t h(x)-h(x)|
        \le
        \mathbb E\,\omega_h(C_RM_t).
\]
Since $M_t\to0$ almost surely as $t\downarrow0$ and
$0\le \omega_h\le 2\|h\|_\infty$, dominated convergence gives $\|P_t h-h\|_\infty\to0 .$ Therefore $(P_t)$ is a strongly continuous positive contraction semigroup on
$C_0(E)$.

Then we verify \eqref{eq:ce-ergodicity}.

\begin{proposition}[{\normalfont\cite[Corollary~3.2]{Sarantsev2017}}]Let $(\Sigma,\mu,R)$ be orthant SRBM data with $\Sigma$ symmetric positive definite and $R$ completely-$\mathcal S$. If there is a symmetric positive definite matrix $H$ with $HR$ a $Z$ matrix and $H\mu<0$ componentwise, then the SRBM is positive recurrent with a unique stationary probability $\pi_0$ and is $V$-uniformly exponentially ergodic: there exist a locally bounded $V\colon E\to[1,\infty)$ and constants $M,\kappa>0$ such that
\[
 \norm{P_t(x,\cdot)-\pi_0}_{\TV}\le M V(x)e^{-\kappa t},
 \qquad x\in E,\ t\ge0.
\]
\end{proposition} 

The matrix $H$ in the statement of the corollary is exactly such a certificate: it is symmetric positive definite, $HR$ is a $Z$ matrix, and $H\mu<0$ componentwise, while $\Sigma$ is positive definite and $R$ is completely-$\mathcal S$. Hence Sarantsev's criterion supplies the stationary probability $\pi_0$ and the recurrence certificate \eqref{eq:ce-ergodicity}.

Thus \cref{thm:singular-block-obstruction} applies.
\end{proof}

Set
\begin{equation}\label{eq:ce-R-family}
R(a,b,c,d)=
\begin{pmatrix}
1&1&a\\
1&1&b\\
c&d&1
\end{pmatrix},
\qquad
\Sigma=I_3,
\qquad
\mu=\begin{pmatrix}-11\\-7\\-11\end{pmatrix},
\end{equation}
and let $\mathcal P$ be the parameter region
\begin{equation}\label{eq:ce-parameter-region}
\begin{gathered}
0<b\le\frac{11}{20},
\qquad
\frac{56b+25}{60}\le a\le\frac{3b+3}{5},\qquad
0<c\le\frac4{25},
\qquad
\frac23\le d\le\frac{35}{46}.
\end{gathered}
\end{equation}
The subset obtained by making all inequalities strict is nonempty, so $\mathcal P$ contains a genuine four-dimensional region.

\begin{theorem}[Four-parameter family]\label{thm:four-parameter-family}
For every $(a,b,c,d)\in\mathcal P$, the SRBM data in \eqref{eq:ce-R-family} define a nonsingular completely $S$, exponentially ergodic SRBM with a unique stationary probability. Its reflection matrix has the singular proper principal block
\[
 R_{\{1,2\},\{1,2\}}
 =\begin{pmatrix}1&1\\1&1\end{pmatrix}.
\]
For every such parameter choice, signed BAR interior uniqueness fails, and the space of zero-mass interior BAR coordinates is infinite-dimensional.
\end{theorem}

\begin{proof}
Every entry of $R(a,b,c,d)$ is positive. Hence every principal submatrix is an $S$ matrix, with the all-ones vector as a witness, and $R$ is completely $S$. A direct calculation gives $\det R(a,b,c,d)=(b-a)(c-d).$ The parameter bounds imply $a>b$ and $d>c$, so the determinant is positive. The block indexed by $A=\{1,2\}$ is singular, with
\[
 v=\begin{pmatrix}1\\-1\end{pmatrix}\in\ker R_{AA},
 \qquad
 w=R_{\{3\},A}v=c-d\ne0.
\]

Consider the symmetric matrix
\begin{equation}\label{eq:ce-H}
H=
\begin{pmatrix}
\frac12&-\frac3{10}&-\frac3{10}\\
-\frac3{10}&\frac7{25}&\frac18\\
-\frac3{10}&\frac18&\frac{23}{100}
\end{pmatrix}.
\end{equation}
Its leading principal minors are $\frac12,\frac1{20},\frac{79}{80000},$ so $H$ is positive definite. Multiplication gives
\begin{equation}\label{eq:ce-HR-family}
HR=
\begin{pmatrix}
\frac15-\frac{3c}{10}
 &\frac15-\frac{3d}{10}
 &\frac a2-\frac{3b}{10}-\frac3{10}\\[2mm]
\frac c8-\frac1{50}
 &\frac d8-\frac1{50}
 &-\frac{3a}{10}+\frac{7b}{25}+\frac18\\[2mm]
\frac{23c}{100}-\frac7{40}
 &\frac{23d}{100}-\frac7{40}
 &-\frac{3a}{10}+\frac b8+\frac{23}{100}
\end{pmatrix}.
\end{equation}
The six off-diagonal entries are nonpositive by \eqref{eq:ce-parameter-region}. Also
\begin{equation}\label{eq:ce-Hmu}
H\mu=
\begin{pmatrix}
-\frac1{10}\\[1mm]
-\frac7{200}\\[1mm]
-\frac{21}{200}
\end{pmatrix}<0.
\end{equation}
\Cref{cor:checkable-completely-S} completes the proof.
\end{proof}

For this entire family the algebraic source is explicit. Let $S=\{(0,0,y):y>0\}$ and choose a nonconstant $\phi\in C_c^\infty((0,\infty))$. Define
\begin{align}
 \int_{F_1}g\,d\zeta_1
 &=\int_0^\infty g(0,0,y)\phi(y)\,dy,
                                                    \label{eq:ce-family-zeta1}\\
 \int_{F_2}g\,d\zeta_2
 &=-\int_0^\infty g(0,0,y)\phi(y)\,dy,
 \qquad \zeta_3=0.                         \label{eq:ce-family-zeta2}
\end{align}
Since $R_1-R_2=(0,0,c-d)^T$,
\begin{equation}\label{eq:ce-family-source}
 \sum_{i=1}^3\int_{F_i}D_i f\,d\zeta_i
 =(d-c)\int_0^\infty\phi'(y)f(0,0,y)\,dy.
\end{equation}
Thus
\begin{equation}\label{eq:ce-family-chi}
 \chi(dx)=(d-c)\phi'(x_3)\,dx_3\,
          \delta_0(dx_1)\delta_0(dx_2)
\end{equation}
is nonzero and has total mass zero. Its zero potential and the corresponding boundary occupation potentials give the signed BAR counterexample.

\begin{corollary}[Concrete rational counterexample]\label{cor:rational-counterexample}
For
\begin{equation}\label{eq:ce-explicit-data}
R=
\begin{pmatrix}
1&1&\frac35\\
1&1&\frac16\\
\frac2{15}&\frac23&1
\end{pmatrix},
\qquad
\Sigma=I_3,
\qquad
\mu=\begin{pmatrix}-11\\-7\\-11\end{pmatrix},
\end{equation}
there is a finite signed BAR tuple $(\overline\pi,\overline\nu_1,\overline\nu_2,\overline\nu_3)$ such that $\overline\pi(E)=0$ and $\overline\pi\ne0$.
\end{corollary}

\begin{proof}
The parameter choice belongs to $\mathcal P$. Exact arithmetic gives
\begin{equation}\label{eq:ce-explicit-certificates}
 \det R=\frac{52}{225}>0,
 \qquad
 R^{-1}\mu=
 \begin{pmatrix}
 -\frac{365}{104}\\[1mm]
 -\frac{203}{104}\\[1mm]
 -\frac{120}{13}
 \end{pmatrix}<0.
\end{equation}
The matrix $H$ in \eqref{eq:ce-H} satisfies
\begin{equation}\label{eq:ce-explicit-HR}
HR=
\begin{pmatrix}
\frac4{25}&0&-\frac1{20}\\[1mm]
-\frac1{300}&\frac{19}{300}&-\frac1{120}\\[1mm]
-\frac{433}{3000}&-\frac{13}{600}&\frac{17}{240}
\end{pmatrix},
\end{equation}
and \eqref{eq:ce-Hmu} holds. Thus all assumptions in \cref{cor:checkable-completely-S} are verified.

For complete explicitness, take the standard bump
\begin{equation}\label{eq:ce-bump}
\phi(y)=
\begin{cases}
\displaystyle
\exp\!\left(-\frac{1}{(y-1)(2-y)}\right),&1<y<2,\\[2mm]
0,&\text{otherwise}.
\end{cases}
\end{equation}
Define $\zeta_1,\zeta_2,\zeta_3$ by \eqref{eq:ce-family-zeta1}--\eqref{eq:ce-family-zeta2}. Here
\[
 R_1-R_2=
 \begin{pmatrix}0\\0\\-\frac8{15}\end{pmatrix},
\]
so
\begin{equation}\label{eq:ce-explicit-source}
 \chi(dx)=\frac8{15}\phi'(x_3)\,dx_3\,
          \delta_0(dx_1)\delta_0(dx_2),
 \qquad
 \chi(E)=0,
 \qquad
 \chi\ne0.
\end{equation}
Let $P_t$ be the reflected semigroup and let $K_i$ be the kernels in \eqref{eq:ce-Ki}. Set
\begin{equation}\label{eq:ce-explicit-potentials}
 \overline\pi=\int_0^\infty\chi P_t\,dt,
 \qquad
 \overline\nu_i=\zeta_i+\sum_{n=0}^\infty(\chi P_n)K_i.
\end{equation}
By \cref{prop:zero-potential-correction}, all measures in \eqref{eq:ce-explicit-potentials} are finite and satisfy
\[
 \int_E Lf\,d\overline\pi
 +\sum_{i=1}^3\int_{F_i}D_i f\,d\overline\nu_i=0,
 \qquad f\in C_b^2(E).
\]
Moreover $\overline\pi(E)=0$ and $\overline\pi\ne0$. Therefore $\overline\pi$ cannot be a scalar multiple of the stationary probability $\pi_0$.
\end{proof}

\section{Unified understanding: RI defects and boundary algebra}\label{sec:unified-understanding}

This section isolates resolvent insertion as the common mechanism behind both the signed BAR uniqueness theorem and the completely-\(\mathcal S\) obstruction.  We view a signed BAR tuple through its interior coordinate and measure the failure of this coordinate to satisfy resolvent insertion by its RI defect.  After quotienting out the stationary BAR direction and the pure boundary kernel, the remaining part of the BAR kernel is exactly the RI defect quotient.  In the Harrison--Reiman nonsingular \(M\)-matrix class this quotient vanishes, while in the completely-\(\mathcal S\) singular-block regime the boundary source and its zero-potential lift produce a nonzero class in this quotient. We denote

%This section records the common structure behind the uniqueness theorem and the completely-$\mathcal S$ obstruction.   The point is to separate two issues that were intertwined in the preceding sections.  The analytic question is whether the interior coordinate of a BAR tuple satisfies the resolvent insertion identity.  The algebraic question is whether the remaining pure boundary part is zero, or whether a boundary gauge can create a centered source.  We keep possible failures of the resolvent insertion as an explicit quotient.

\[
 \mathsf M_E=\cM(E),
 \qquad
 \mathsf M_\partial=\prod_{i=1}^d \cM(F_i),
 \qquad
 \mathcal T=C_b^2(E).
\]
For $(\pi,\nu)\in \mathsf M_E\times\mathsf M_\partial$, define the BAR functional $\mathcal A(\pi,\nu)(f)
 =\int_E Lf\,d\pi+
   \sum_{i=1}^d\int_{F_i}D_i f\,d\nu_i,$ for $f\in\mathcal T$. Thus finite signed BAR tuples are precisely $\ker\mathcal A$.  Define the pure boundary operator $\partial_R\zeta(f)
 =\sum_{i=1}^d\int_{F_i}D_i f\,d\zeta_i$. Since one can identify a finite signed measure $\chi\in\cM(E)$ with the functional $f\mapsto\int_E f\,d\chi$ on $\mathcal T$, we define the relation $\partial_R\zeta=\chi$ by the condition that $ \sum_{i=1}^d\int_{F_i}D_i f\,d\zeta_i
 =\int_E f\,d\chi$ holds for all $f\in\mathcal T$. 

\subsection{The RI-defect quotient}

Let $Z_{\rm BAR}:=\ker\mathcal A$ and $\Pi_E(\pi,\nu)=\pi$, and define the space of interior coordinates of signed BAR tuples by
\begin{equation}\label{eq:unified-I-BAR}
 \mathcal I_{\rm BAR}:=\Pi_E Z_{\rm BAR}
 =\{\pi\in\cM(E):\text{ there exists }\nu\in\mathsf M_\partial
       \text{ with }(\pi,\nu)\in\ker\mathcal A\}.
\end{equation}
For a finite signed measure $m\in\cM(E)$, define its resolvent-insertion defect by
\begin{equation}\label{eq:unified-r-defect}
 \mathfrak r(m)(\lambda,h)
 :=\int_E(\lambda R_\lambda h-h)\,dm,
 \qquad \lambda>0,\quad h\in C_0(E).
\end{equation}
Thus $m$ satisfies the resolvent identity precisely when $\mathfrak r(m)=0$.  Set $\mathcal I_{\rm RI}:=\mathcal I_{\rm BAR}\cap\ker\mathfrak r,
$ and define the RI-defect quotient $ \mathcal Q_{\rm RI}:=\mathcal I_{\rm BAR}/\mathcal I_{\rm RI}.
$ Equivalently, $ \mathcal Q_{\rm RI}\simeq \mathfrak r(\mathcal I_{\rm BAR}).
$ This quotient records exactly the part of the signed BAR kernel not killed by resolvent insertion.

\begin{lemma}[BAR quotient by RI defects]\label{lem:RI-defect-quotient}
Assume that the reflected semigroup is $C_0$-Feller and strongly continuous, that it has a unique invariant probability $\pi_0$, and that $s_0=(\pi_0,\nu^0)$ is the stationary BAR tuple.  Then there is a canonical exact sequence
\begin{equation}\label{eq:unified-exact-sequence}
 0
 \longrightarrow
 \ker\partial_R
 \longrightarrow
 \ker\mathcal A/\mathbb R s_0
 \longrightarrow
 \mathcal Q_{\rm RI}
 \longrightarrow
 0.
\end{equation}
Equivalently,
\begin{equation}\label{eq:unified-quotient-isomorphism}
 \frac{\ker\mathcal A}
 {\mathbb R s_0\oplus(\{0\}\times\ker\partial_R)}
 \simeq
 \mathcal Q_{\rm RI}.
\end{equation}
\end{lemma}

\begin{proof}
Let $ H_{\rm BAR}:=\ker\mathcal A/\mathbb R s_0.$ Define $\Gamma:H_{\rm BAR}\to\mathcal Q_{\rm RI},$ and $\Gamma([(\pi,\nu)])=[\pi].$  Replacing $(\pi,\nu)$ by $(\pi,\nu)+c(\pi_0,\nu^0)$ changes the interior coordinate by $c\pi_0$.  Since $\pi_0$ is invariant,
\[
 \int_E R_\lambda h\,d\pi_0
 =\int_0^\infty e^{-\lambda t}\int_E P_t h\,d\pi_0\,dt
 =\lambda^{-1}\int_E h\,d\pi_0,
\]
so $\mathfrak r(\pi_0)=0$ and $\pi_0\in\mathcal I_{\rm RI}$.  Hence $[\pi]$ is unchanged in $\mathcal Q_{\rm RI}$.  The map $\Gamma$ is surjective by the definition of $\mathcal Q_{\rm RI}$.

We compute its kernel.  Suppose $\Gamma([(\pi,\nu)])=0$.  Then $\pi\in\mathcal I_{\rm RI}$, so
\[
 \int_E(\lambda R_\lambda h-h)\,d\pi=0,
 \qquad h\in C_0(E),\quad \lambda>0.
\]
By \cref{prop:semigroup-invariance}, $\pi P_t=\pi$ for all $t\ge0$.  By \cref{prop:interior-identification}, we have $\pi=c\pi_0,$ where $c=\pi(E).$ Since both $(\pi,\nu)$ and $c(\pi_0,\nu^0)$ are BAR tuples, $(0,\nu-c\nu^0)
 =(\pi,\nu)-c(\pi_0,\nu^0)
 \in\ker\mathcal A.$ Equivalently, $\partial_R(\nu-c\nu^0)=0.$ Thus $[(\pi,\nu)]=[(0,\eta)]$ for some $\eta\in\ker\partial_R$.

Conversely, if $\eta\in\ker\partial_R$, then $(0,\eta)\in\ker\mathcal A$ and its interior coordinate has zero RI defect.  Hence $\ker\Gamma=\{[(0,\eta)]:\eta\in\ker\partial_R\}.$

The map $\eta\mapsto[(0,\eta)]$ is injective because if $[(0,\eta)]=0$ in $\ker\mathcal A/\mathbb R s_0$, then $(0,\eta)=c(\pi_0,\nu^0)$ for some $c\in\mathbb R$; the interior coordinate gives $c\pi_0=0$, hence $c=0$ and $\eta=0$.  This proves the exact sequence \eqref{eq:unified-exact-sequence}.  The quotient isomorphism \eqref{eq:unified-quotient-isomorphism} is the corresponding first-isomorphism statement.  The sum in the denominator is direct by the same interior-coordinate argument.
\end{proof}

\begin{remark}[Strength of the $C_0$ Feller Assumption]\label{rem:C0-Feller-strength} The only assumption we make in this abstraction is the $C_0$ Feller property used in the RI defect quotient which is only a soft
semigroup input instead of a boundary regularity assumption.  It means that
\[
        P_t C_0(E)\subset C_0(E),
        \qquad
        \|P_t h-h\|_\infty\to0
        \quad\text{as }t\downarrow0,\ h\in C_0(E).
\]
It is much weaker than strong Feller smoothing, the existence of transition
densities, differentiability of $P_t h$, or closed domain $C^2$ regularity of
the probabilistic resolvent.  In particular, it does not assert that
$R_\lambda h$ has a classical oblique Neumann trace on the boundary.

In the present argument this input is used only through standard semigroup
consequences.  In \cref{prop:semigroup-invariance}, it makes $ t\longmapsto \int_E P_t h\,d\bar\pi$ bounded and continuous for \(h\in C_0(E)\).  The Laplace transform identity
then upgrades an almost everywhere conclusion to equality for every
\(t\ge0\), using uniqueness of the Laplace transform
\cite[Chapter~II]{Doetsch1974}.  In \cref{prop:zero-potential-correction},
strong continuity is used through the density of the generator domain in
\(C_0(E)\), a standard fact for strongly continuous Feller semigroups
\cite[Chapter~1, Section~2]{EthierKurtz1986}.  Thus the exact sequence in
the unified section could be formulated with these consequences directly:
RI null interior coordinates must be signed invariant measures, and the
zero potential must be injective on the centered source class.

For the positive Harrison--Reiman nonsingular \(M\) matrix class, the
\(C_0\) Feller property is not an extra boundary smoothness input.
It is proved in \cref{thm:regularity-SRBM}.  The proof uses the synchronous
Lipschitz estimate for the Skorokhod map, imported from
\cite[Proposition~2.6]{LipshutzRamanan2019}, together with Brownian path
continuity.  The same mechanism applies more generally to any orthant SRBM
for which one has a pathwise unique continuous construction and a finite
horizon Lipschitz estimate of the form $\sup_{0\le s\le T}|Z_s^x-Z_s^y|
        \le C_T |x-y|$ under a synchronous coupling.  Then \(x\mapsto P_t h(x)\) is continuous,
\(P_t h\) vanishes at infinity, and \(\|P_t h-h\|_\infty\to0\) follow by
the same argument as in the proof of \cref{thm:regularity-SRBM}.

For the completely \(\mathcal S\) counterexample regime, the matrix
condition \(R\) completely \(\mathcal S\) should not be read as a substitute
for the analytic inputs in the RI reduction.  It is an existence geometry
condition.  In the checkable subclass of \cref{cor:checkable-completely-S},
Taylor--Williams provide the orthant SRBM existence and Feller framework for
completely \(\mathcal S\) reflection data \cite{TaylorWilliams1993}, while
Sarantsev's Lyapunov criterion supplies the stationary probability and the
\(V\) uniform total variation exponential ergodicity used in
\eqref{eq:ce-ergodicity} \cite[Corollary~3.2]{Sarantsev2017}.  The
one step regulator bound is then checked directly in
\cref{cor:checkable-completely-S}.  These assumptions are sufficient for the
zero potential construction, but they do not imply the global RI reduction
for all signed BAR tuples.  Indeed, in the singular block case the boundary
gauge and zero potential construction produce a nonzero RI defect class.
Thus the completely \(\mathcal S\) hypothesis alone is not a uniqueness
regularity assumption.

The minimal zero potential input is also source specific.  For a given
centered source \(\chi\), it is enough to have finite signed measures
\(U\chi\) and \(\Theta_i\chi\) satisfying $\mathcal A(U\chi,\Theta\chi)=-\chi,$ and  $(U\chi)(E)=0$ together with injectivity \(U\chi=0\Rightarrow \chi=0\) on the source class
under consideration.  The quantitative recurrence and regulator assumptions
in \cref{prop:zero-potential-correction} are a convenient sufficient package:
they imply total variation convergence of
\(
        \int_0^\infty \chi P_t\,dt
        \text{ and }
        \sum_{n\ge0}(\chi P_n)K_i,
\)
justify the Poisson identity, and give the injectivity needed to show that a
nonzero centered source yields a nonzero zero mass interior BAR coordinate.
\end{remark}

\paragraph{The positive theorem in the Harrison--Reiman Class} In the Harrison--Reiman nonsingular $M$-matrix setting, \cref{thm:resolvent-insertion} says that every signed BAR tuple satisfies the resolvent identity.  Equivalently, $\mathfrak r(\mathcal I_{\rm BAR})=0$ and $\mathcal Q_{\rm RI}=0.$ Then \cref{lem:RI-defect-quotient} reduces the quotient $\ker\mathcal A/\mathbb R s_0$ to the pure boundary kernel.  The latter is killed by \cref{prop:pure-boundary}, \emph{i.e.} $\ker\partial_R=\{0\}.$ Consequently $\ker\mathcal A=\mathbb R s_0,$ which is the signed uniqueness conclusion of \cref{thm:main}.

This formulation separates the two uses of active-block invertibility.  First, $R_{AA}^{-1}$ defines the active projection
\[
 \mathsf L_Av=v-R_A R_{AA}^{-1}v_A,
 \qquad
 \mathsf L_A R_i=0,
 \quad i\in A.
\]
This is the algebraic input behind the measure--Neumann resolvent insertion.  Second, $R_{AA}^{-T}$ prescribes active oblique jets on $S_A$.  If the active normal gradient is $a$, then $(D_i f|_{S_A})_{i\in A}=R_{AA}^T a,$ so invertibility of $R_{AA}$ permits the choice $a=R_{AA}^{-T}\psi$.  The induction over strata in \cref{prop:pure-boundary} uses exactly this prescription.

\paragraph{The completely-$\mathcal S$ obstruction.} In this section, we advance our understanding of the zero potential construction as a device for cancelling the source term in the BAR but as a way of placing explicit nonzero elements into $\mathcal Q_{\rm RI}$. Starting from the singular boundary gauge, one has a boundary source identity $\partial_R \zeta = \chi .$ Thus the boundary gauge alone has BAR defect \(+\chi\).  The zero potential $Q_\chi=(U_\chi,\Theta_\chi), \text{ where }
U_\chi=\int_0^\infty \chi P_t\,dt,$ is constructed so that its BAR contribution is exactly the opposite defect: $A(Q_\chi)=-\chi .$ Therefore we have $Q_\chi+(0,\zeta)=(U_\chi,\Theta_\chi+\zeta)\in \ker A,$ so the zero potential turns the boundary gauge into a genuine signed BAR tuple.  However, this cancellation does not make the source \(\chi\) disappear.  Instead, we shows that \(\chi\) reappears as the resolvent insertion defect of the interior measure \(U_\chi\): $r(U_\chi)(\lambda,h)=-\int_E R_\lambda h\,d\chi .$ Hence, if \(\chi\neq 0\), then \(U_\chi\) cannot satisfy the resolvent insertion identity; otherwise the strong continuity of the semigroup would imply that \(\chi\) vanishes on all functions in \(C_0(E)\), forcing \(\chi=0\).  Consequently, the zero potential is the mechanism that converts the singular boundary source into a concrete nonzero class. $0\neq [U_\chi]\in Q_{\mathrm{RI}} .$ This is why the completely \(S\) counterexample is best understood as an RI defect: the boundary algebra creates the centered source \(\chi\), and the zero potential lifts that source into a genuine BAR tuple whose interior coordinate carries a nonzero resolvent insertion defect.

\begin{proposition}[Boundary sources give RI-defect classes]\label{prop:boundary-source-RI-defect}
Assume the semigroup is strongly continuous on $C_0(E)$ and that the zero-potential construction of \cref{prop:zero-potential-correction} is available for a centered source $\chi$.  Write
\[
 Q\chi=(U\chi,\Theta\chi),
 \qquad
 U\chi=\int_0^\infty\chi P_t\,dt,
 \qquad
 \Theta_i\chi=\sum_{n=0}^\infty(\chi P_n)K_i.
\]
If $\zeta\in\mathsf M_\partial$ satisfies $\partial_R\zeta=\chi$, then
\begin{equation}\label{eq:unified-lifted-BAR-tuple}
 Q\chi+(0,\zeta)
 =(U\chi,\Theta\chi+\zeta)
 \in\ker\mathcal A.
\end{equation}
Its image under the map in \cref{lem:RI-defect-quotient} is the class $[U\chi]\in\mathcal Q_{\rm RI}$, and
\begin{equation}\label{eq:unified-defect-of-potential}
 \mathfrak r(U\chi)(\lambda,h)
 =-\int_E R_\lambda h\,d\chi,
 \qquad \lambda>0,\quad h\in C_0(E).
\end{equation}
In particular, if $\chi\ne0$, then $[U\chi]\ne0$ in $\mathcal Q_{\rm RI}$.
\end{proposition}

\begin{proof}
By \cref{prop:zero-potential-correction} and fact that $\partial_R\zeta=\chi$, thus $\mathcal A(Q\chi)(f)=-\int_E f\,d\chi$ and $\mathcal A(0,\zeta)(f)=\int_E f\,d\chi$ holds for all $f\in\mathcal T$. Adding the two identities gives \eqref{eq:unified-lifted-BAR-tuple}.  Hence $U\chi\in\mathcal I_{\rm BAR}$ and its class in the RI-defect quotient is $[U\chi]$.

It remains to identify its defect.  For $h\in C_0(E)$ put $F(t)=\int_E P_t h\,d\chi.$ The total-variation convergence in \cref{prop:zero-potential-correction} justifies the following Fubini calculation:
\begin{align*}
 \mathfrak r(U\chi)(\lambda,h)
 &=\int_0^\infty
    \int_E P_t(\lambda R_\lambda h-h)\,d\chi\,dt  =\int_0^\infty
    \left(\lambda\int_0^\infty e^{-\lambda s}F(t+s)\,ds-F(t)\right)dt  \\
 &=\int_0^\infty F(u)(1-e^{-\lambda u})\,d u
   -\int_0^\infty F(u)\,du  =-\int_0^\infty e^{-\lambda u}F(u)\,du
  =-\int_E R_\lambda h\,d\chi.
\end{align*}
If $[U\chi]=0$ in $\mathcal Q_{\rm RI}$, then $\mathfrak r(U\chi)=0$, hence $\chi(R_\lambda h)=0$ for all $\lambda>0$ and $h\in C_0(E)$.  Since the semigroup is strongly continuous on $C_0(E)$,
\[
 \lambda R_\lambda h
 =\int_0^\infty e^{-s}P_{s/\lambda}h\,ds
 \longrightarrow h
 \quad\text{in }C_0(E)
\]
as $\lambda\to\infty$.  Therefore $\chi(h)=0$ for every $h\in C_0(E)$, and the finite Radon measure $\chi$ is zero.  Thus $\chi\ne0$ implies $[U\chi]\ne0$.
\end{proof}

Suppose that $R$ is nonsingular and that $R_{AA}$ is singular for some nonempty proper subset $A\subsetneq J$.  Put $T=J\setminus A$ and choose $0\ne v\in\ker R_{AA},$ where $w=R_{TA}v.$ Since $R$ is nonsingular, $w\ne0$.  On $S_A$, the local boundary symbol is
\begin{equation}\label{eq:unified-local-symbol}
 \sum_{i\in A}v_iD_i f
 =(R_{AA}v)\cdot\nabla_A f+(R_{TA}v)\cdot\nabla_T f.
\end{equation}
Because $R_{AA}v=0$, only the tangential derivative $w\cdot\nabla_T f$ remains.  \Cref{prop:singular-block-gauge} turns this symbol calculation into a boundary source: for a compactly supported smooth density $\phi$ on the open stratum with $w\cdot\nabla_T\phi\not\equiv0$, it constructs boundary measures $\zeta_i$ and a finite nonzero centered measure $\chi$ supported on $S_A$ such that $\partial_R\zeta=\chi.$ Under the recurrence and regulator hypotheses of \cref{prop:zero-potential-correction}, \cref{prop:boundary-source-RI-defect} gives $(U\chi,\Theta\chi+\zeta)\in\ker\mathcal A$ and places its interior coordinate into the RI-defect quotient as $0\ne[U\chi]\in\mathcal Q_{\rm RI}.$ Thus signed BAR uniqueness fails.  In fact, the singular-block construction does more than produce a BAR tuple with zero total interior mass: it produces a concrete nonzero resolvent-insertion defect, $\mathfrak r(U\chi)(\lambda,h)
 =-\chi(R_\lambda h).$ Therefore a global resolvent insertion theorem cannot hold in this singular-block regime once the zero-potential lift is available.

\section*{Acknowledgment}

The authors would like to thank Jose Blanchet for bringing this open problem to our attention and encouraging us to pursue an AI-based solution. We are also grateful to Jose Blanchet, Yufan Chen and Wenhao Yang for their valuable feedback on this manuscript. The authors are also grateful to Bin Dong, Xiao Ma, Jiajin Li and Jianfeng Lu for their insightful discussions regarding the application of AI in mathematical proving and the formulation of our AI usage disclosure.

\begin{appendix}

\section{A stable example with \texorpdfstring{$g=R_\lambda h\notin C^2(E)$}{g=Rlambda h not in C2(E)}}\label{app:resolvent-not-C2}

This appendix gives a concrete example in which the probabilistic resolvent \(g=R_\lambda h\) is not in \(C^2(E)\).  The point is that interior smoothness and one-sided oblique flatness on open faces do not guarantee closed-domain \(C^2\) regularity at a corner.  We choose a smooth nonnegative source \(h\in C_c^\infty(E^\circ)\) such that \(h\not\equiv0\) but \(h(0)=0\). 

We first justify that \(g(0)>0\). Since \(h\ge0\), \(h\not\equiv0\), and
\(\operatorname{supp}h\subset E^\circ\), choose \(z_\ast\in E^\circ\), \(r>0\), and
\(c_h>0\) such that \(\overline{B(z_\ast,r)}\subset E^\circ\) and
\(h\ge c_h\) on \(\overline{B(z_\ast,r)}\). Let \(\Gamma\) be the Harrison--Reiman
Skorokhod map. On \([0,2]\), \(\Gamma\) is Lipschitz with constant \(K_\Gamma\).
Define
\[
        \gamma(t)=tz_\ast,\quad 0\le t\le1,
        \qquad
        \gamma(t)=z_\ast,\quad 1\le t\le2.
\]
Since \(\gamma\) stays in \(E\), \(\Gamma(\gamma)=\gamma\). For the SRBM started from
zero, the free input is \(X(t)=\mu t+W(t)\). Put \(b(t)=\gamma(t)-\mu t\). Since
Brownian motion has full support in \(C_0([0,2],\mathbb R^3)\), the event $A=\left\{
        \sup_{0\le t\le2}|W(t)-b(t)|<r/K_\Gamma
        \right\}$ has positive probability. On \(A\), the free input \(X\) is within \(r/K_\Gamma\) of
\(\gamma\), and therefore the reflected path \(Z^0=\Gamma(X)\) is within \(r\) of
\(\gamma\). Since \(\gamma(t)=z_\ast\) for \(1\le t\le2\), we have
\(Z^0(t)\in B(z_\ast,r)\) throughout \([1,2]\) on \(A\). Hence $g(0)\ge
        \mathbb P(A)c_h\int_1^2 e^{-\lambda t}\,dt
        >
        0 .$

On the other hand, if \(g\) were \(C^2\) up to the corner and satisfied the exact face conditions, those conditions would force \(\nabla g(0)=0\) and \(D^2g(0)=0\).  The resolvent equation at the corner would then give \(0=h(0)=(\lambda-L)g(0)=\lambda g(0)\), contradicting \(g(0)>0\).

Consider $d=3$, $\Sigma=I_3$, and
\[
R=
\begin{pmatrix}
1&0&-\frac12\\
-\frac12&1&0\\
0&-\frac12&1
\end{pmatrix},
\qquad
R^{-1}=\frac17
\begin{pmatrix}
8&2&4\\
4&8&2\\
2&4&8
\end{pmatrix}\ge0 .
\]
Thus $R$ is a nonsingular $M$-matrix.  Let $\mu=-R\mathbf 1$; then $R^{-1}\mu=-\mathbf 1<0$, so the data are stable in the sense of \eqref{eq:stability}.

Choose $h\in C_c^\infty(E^\circ)$ with $h\ge0$ and $h\not\equiv0$, and set $g=R_\lambda h$ for some $\lambda>0$.  Since the Brownian input has full support on compact time intervals and the Harrison--Reiman Skorokhod map is continuous, the SRBM started from the origin has positive probability of entering a ball on which $h>0$ and then remaining there for a nonzero time interval.  Hence $g(0)=\mathbb E_0\int_0^\infty e^{-\lambda t}h(Z_t)\,dt>0$.  Also $h(0)=0$, because $h$ is supported in the interior.

Assume, for contradiction, that $g\in C^2(E)$ in the closed-domain sense.  On each open face $F_i^\circ$, the one-sided derivative identity of \cref{prop:resolvent-directional} applies in the feasible direction $R_i$.  Since $\mathsf L_xR_i=0$ for $x\in F_i^\circ$, it gives $D_i g=0$ on $F_i^\circ$.  By continuity of the first derivatives, these identities extend to the origin as $R^T\nabla g(0)=0$.  Since $R$ is invertible, $\nabla g(0)=0$.

Let $H=D^2g(0)$.  For $j\ne i$, differentiating $D_i g=R_i\cdot\nabla g=0$ in the tangential direction $e_j$ along $F_i^\circ$ and then letting the tangential point tend to the origin gives $R_i^THe_j=0$.  Equivalently, $\operatorname{offdiag}(R^TH)=0$.  Thus $R^TH=\diag(d_1,d_2,d_3)$ for some real numbers $d_1,d_2,d_3$, and therefore
\[
        H=R^{-T}\diag(d_1,d_2,d_3)
        =\frac17
        \begin{pmatrix}
        8d_1&4d_2&2d_3\\
        2d_1&8d_2&4d_3\\
        4d_1&2d_2&8d_3
        \end{pmatrix}.
\]
Since $H$ is symmetric, comparison of the $(1,2)$, $(2,3)$, and $(1,3)$ entries gives $d_1=2d_2$, $d_2=2d_3$, and $d_3=2d_1$.  Hence $d_1=d_2=d_3=0$, so $H=0$.

The interior resolvent equation gives $(\lambda-L)g=h$ on $E^\circ$.  If $g\in C^2(E)$, the left-hand side extends continuously to the origin.  Since $h(0)=0$, $\nabla g(0)=0$, $D^2g(0)=0$, and $\Sigma=I_3$, this gives $0=h(0)=(\lambda-L)g(0)=\lambda g(0)$, contradicting $g(0)>0$.  Consequently, for this stable Harrison--Reiman SRBM and this smooth compactly supported interior source, $R_\lambda h\notin C^2(E)$.
\end{appendix}

\bibliographystyle{imsart-number}
\bibliography{references}
\end{document}